\documentclass[11pt,a4paper]{article}

\usepackage[ngerman,english]{babel}
\usepackage[utf8]{inputenc}

\usepackage[DIV=12]{typearea}

\usepackage{xcolor}

\usepackage{xparse}
\usepackage{amsmath, amsfonts, amssymb, amsthm}
\usepackage{mathrsfs}
\usepackage{mathtools}
\usepackage{enumitem}

\usepackage[noadjust]{cite}

\usepackage[capitalise]{cleveref}

\crefname{equation}{}{}
\crefname{enumi}{}{}

\newcommand{\R}{\mathbb{R}}

\newcommand{\N}{\mathbb{N}}

\newcommand{\ee}{\mathrm{e}}

\DeclareDocumentCommand\dd{ o g d() }{
	\IfNoValueTF{#2}{
		\IfNoValueTF{#3}
			{\mathrm{d}\IfNoValueTF{#1}{}{^{#1}}}
			{\mathinner{\mathrm{d}\IfNoValueTF{#1}{}{^{#1}}\argopen(#3\argclose)}}
		}
		{\mathinner{\mathrm{d}\IfNoValueTF{#1}{}{^{#1}}#2} \IfNoValueTF{#3}{}{(#3)}}
	}

\newcommand{\dy}{\dd{y}}

\newcommand{\dr}{\dd{r}}

\newcommand{\del}{\partial}
\newcommand{\eps}{\varepsilon}
\newcommand{\im}{\mathrm{i}}

\newcommand{\FL}[1]{(-\Delta)^{#1}}

\newcommand{\F}{\mathcal{F}}

\newcommand{\power}{\mathfrak{s}} 

\newcommand{\X}{\mathcal{X}}
\newcommand{\Y}{\mathcal{Y}}

\DeclareMathOperator{\Res}{Res}
\DeclareMathOperator{\Id}{Id}

\newcommand{\Ord}{\mathcal{O}}

\newcommand{\vcc}{\vcentcolon}

\DeclarePairedDelimiter\abs{\lvert}{\rvert}
\DeclarePairedDelimiter\norm{\Vert}{\rVert}

\theoremstyle{plain}
\newtheorem{theorem}{Theorem}[section]
\newtheorem{lemma}[theorem]{Lemma}
\newtheorem{proposition}[theorem]{Proposition}

\theoremstyle{definition}

\theoremstyle{remark}
\newtheorem{remark}[theorem]{Remark}

\numberwithin{equation}{section}

 \title{The Amplitude Equation for the Space-Fractional Swift-Hohenberg Equation}

\author{Christian Kuehn\thanks{Technical University of Munich, Department of Mathematics, School of Computation Information and Technology, 85748 Garching b.~M\"unchen, Germany, \texttt{ckuehn@ma.tum.de}} { and} Sebastian Throm\thanks{Ume{\aa} University,  Department of Mathematics and Mathematical Statistics, 90187 Ume{\aa},  \texttt{sebastian.throm@umu.se}}}

\date{}

\begin{document}

\maketitle

\begin{abstract}
Non-local reaction-diffusion partial differential equations (PDEs) involving the fractional Laplacian have arisen in a wide variety of applications. One common tool to analyse the dynamics of classical local PDEs near instability is to derive local amplitude/modulation approximations, which provide local normal forms classifying a wide variety of pattern-formation phenomena. In this work, we study amplitude equations for the space-fractional Swift-Hohenberg equation. The Swift-Hohenberg equation is a basic model problem motivated by pattern formation in fluid dynamics and has served as one of the main PDEs to develop general techniques to derive amplitude equations. We prove that there exists near the first bifurcation point an approximation by a (real) Ginzburg-Landau equation. Interestingly, this Ginzburg-Landau equation is a local PDE, which provides a rigorous justification of the physical conjecture that suitably localized unstable modes can out-compete superdiffusion and re-localize a PDE near instability. Our main technical contributions are to provide a suitable function space setting for the approximation problem, and to then bound the residual between the original PDE and its amplitude equation.  
\end{abstract}

\section{Introduction}
\label{sec:intro}

Anomalous diffusion has been studied in a wide variety of contexts recently. One important case is superdiffusion, where the mean square displacement of a process scales superlinearly in time~\cite{MetzlerKlafter}. On a macroscopic level of observables/densities, the resulting operator is the fractional Laplacian $\FL{\power/2}$ for $\power\in(0,2)$. Due to this physical motivation, the mathematical analysis of fractional reaction-diffusion partial differential equations (PDEs) has to be developed~\cite{bucur2016nonlocal,CaffarelliSilvestre,ChenKimSong,MainardiLuchkoPagnini}. Several directions to understand the dynamics of fractional reaction-diffusion equations have been studied, e.g., travelling waves~\cite{AchleitnerKuehn1,CabreRoquejoffre,del-Castillo-NegreteCarrerasLynch}, numerical methods~\cite{AKMR21,BurrageHaleKay,faustmann-melenk-praetorius21}, and stability/spectral theory~\cite{EhstandKuehnSoresina,golovin2008turing,servadei2014spectrum}; for a more detailed recent survey and introduction to the field of dynamics of fractional reaction-diffusion equations we refer to~\cite{Achleitneretal}. In this work, we contribute to the mathematical analysis of fractional reaction-diffusion PDEs and study the (quadratic/cubic) fractional Swift-Hohenberg equation 
\begin{equation}
\label{eq:SH:qcintro}
 \del_{t}u=-(1-\FL{\power/2})^{2}u+pu-a_1 u^2 - a_2 u^3, \qquad u(x,0)=u_0(x),
\end{equation}
for $u=u(x,t)\in\R$, $x\in\R$, $p\in\R$ is a parameter, and $t\in[0,T_{\textnormal{f}}]$ up to a suitable final time $T_{\textnormal{f}}>0$. The classical Swift-Hohenberg equation for $\power=2$, $a_1=0$, and $a_2=1$ has been studied intensively, and it serves as a key model problem to test amplitude/modulation equations that are obtained as normal forms near bifurcation points~\cite{CrossHohenberg,Hoyle,KuehnBook1,ScU17}. Indeed, one easily checks that the classical Swift-Hohenberg equation has a steady state $u_*(x)\equiv 0$ for all $p\in\R$. This trivial branch of solutions $\{(u_*,p)=(0,0)|p\in\R\}$ changes stability from attracting to repelling at the bifurcation point $p_{\textnormal{c}}=0$. In a suitable neighbourhood of $p_{\textnormal{c}}$ of size $\mathcal{O}(\eps^2)$ for a small parameter $\eps>0$ one hopes to find a simplified PDE, analogous to normal form theory for finite-dimensional bifurcations, that approximates the classical Swift-Hohenberg equation. Let us set $p=\eps^2$, then one may prove that the solutions of the classical Swift-Hohenberg equation can be well-approximated using a decomposition 
\begin{equation}
\label{eq:approxGL}
 \psi(x,t)=\eps (A(X,T)\textnormal{e}^{\textnormal{i}x}+\bar{A}(X,T)\textnormal{e}^{-\textnormal{i}x}),\quad T=\eps^2 t,~X=\eps x,
\end{equation}
into a leading-order Fourier mode that has a space-time dependent (complex-valued) amplitude $A=A(X,T)$. This amplitude satisfies an amplitude/modulation equation, which turns out to be the real Ginzburg-Landau equation
\begin{equation}
\label{eq:rGL:classical}
 \del_{t}A=4\partial_x^2A+A-3A|A|^2,
\end{equation}
where one can also provide an approximation order in various norms, e.g., the point-wise estimate $|u(x,t)-\psi(x,t)|\lesssim\eps^2$ holds for all $x\in\R$ on a time domain $t\in(0,T_{\textnormal{f}}/\eps^2]$ if it holds initially at $t=0$. For the fractional Swift-Hohenberg equation~\eqref{eq:SH:qc} it is still relatively straightforward to check by linearising around the trivial branch and using the Fourier transform that the steady state $u_*(x)\equiv 0$ destabilises upon increasing $p$ at $p_{\textnormal{c}}=0$. So we set $p=\eps^2$ and from the viewpoint of analysis it is necessary to ask, whether the fractional Swift-Hohenberg equation
\begin{equation}
\label{eq:SH:qc}
 \del_{t}u=-(1-\FL{\power/2})^{2}u+\eps^2 u-a_1 u^2 - a_2 u^3, \qquad u(x,0)=u_0(x),
\end{equation}
also admits an amplitude/modulation approximation? In this work, we answer this question positively: we compute the resulting amplitude equation of Ginzburg-Landau-type and we prove an approximation result. A precise statement of the approximation result can be found in Theorem~\ref{Thm:main}. Although the fractional Laplacian is a non-local operator, the resulting Ginzburg-Landau equation turns out to be a local PDE, where the exponent of the fractional Laplacian enters as a parameter. Interestingly, for other classes of non-local PDEs with integral convolution operators occurring in the non-linear terms, one also finds local Ginzburg-Landau-type PDEs~\cite{KuehnThrom,MorganDawes}. Hence, these results lend rigorous mathematical support to the physically motivated conjecture that amplitude equations for non-local PDEs should be local PDEs as long as the critical modes provide enough spatial localization of the important dynamics near the bifurcation point.\medskip

The paper is structured as follows: In the remaining part of Section~\ref{sec:intro}, we are going to introduce some background and notation for the fractional Laplacian and we are going to state our main result in Theorem~\ref{Thm:main}, which shows the approximation of the space-fractional Swift-Hohenberg equation via a Ginzburg-Landau PDE locally near the first bifurcation point. We are building our proof upon techniques for deriving amplitude/modulation equations using residual-based error estimates, which we state at the end of Section~\ref{sec:intro}. Then we proceed to introduce the strategy of the proof in Section~\ref{sec:strategy} and provide a suitable semi-group setting for our problem in Bessel potential spaces. Section~\ref{sec:estimates} is devoted to controlling the difference of the non-linearity between the solution to the fractional Swift-Hohenberg equation and its approximation. In Section~\ref{Sec:fract:SW}, as a preparatory step, we collect several properties and estimates for the fractional Swift-Hohenberg operator and the fractional Laplacian which we will finally use in Section~\ref{Sec:residuum} to compute and estimate the residuum of the approximation.

\subsection{Fractional Laplacian and Bessel potential spaces}

We will switch between several equivalent definitions/representations of the fractional Laplacian $(-\Delta)^{\power/2}$ with $\power\in (0,2)$, i.e.\@ in one space dimension we have
\begin{equation}\label{eq:def:frac:Lap}
 \begin{aligned}
  (\FL{\power/2}u)(x)&=c_{\power}P.V.\int_{\R}\frac{u(x)-u(y)}{\abs{x-y}^{1+\power}}\dy\\
  (\FL{\power/2}u)(x)&=\F^{-1}(\abs{\xi}^{\power}(\F u)(\xi)),
 \end{aligned}
\end{equation}
where $c_{\power}$ is a constant depending upon $\power$. Furthermore, $P.V.$ denotes principal value (which we skip in the following for simplicity) while $\F$ refers to the Fourier transform defined via
\begin{equation*}
 \F(u)(\xi)\vcc=\frac{1}{\sqrt{2\pi}}\int_{\R}u(x)\ee^{-\im x \xi}\dd{x}.
\end{equation*}
Alternatively, we might, when it is more convenient, also use the notation $\hat{u}\vcc=\F(u)$ to denote the Fourier transform.

We recall from \cite{ScU17} the following definitions of weighted $L^2$ and fractional Sobolev/Bessel potential spaces respectively. Let $w(x)=(1+\abs{x}^{2})^{1/2}$ be the standard weight function, then we define
\begin{equation*}
 \begin{aligned}
   L_{\theta}^{2}&\vcc=\{u\in L^{2}\;|\; \norm{u}_{L_{\theta}^{2}}\vcc=\norm{uw^{\theta}}_{L^{2}}\}\\
   H^{\theta}&\vcc=\{u\in L^{2}\;|\; \F u\in L_{\theta}^{2}\}.
 \end{aligned}
\end{equation*}
The norm in $H^{\theta}$ is then given by $\norm{u}_{H^{\theta}}=\norm{\F u}_{L_{\theta}^{2}}$.
\begin{remark}\label{Rem:equivalence:Sobolev:norms}
 If $\theta\in\N_0$ it is well-known that $H^{\theta}=\{u\in L^2\mid D^\alpha u\in L^2 \text{ for all } \alpha \text{ with }\abs{\alpha}\leq \theta\}$ where $D^\alpha u$ is the weak/distributional derivative with respect to the multiindex $\alpha$ (see e.g.\@ \cite{Eva10,DPV12}). Moreover, if $\theta\not\in \N_0$ it holds (see e.g.\@ \cite{DPV12})
 \begin{equation*}
  H^{\theta}=\{u\in H^{\lfloor\theta\rfloor}\mid D^{\alpha}u\in H^{\theta-\lfloor \theta\rfloor} \text{ for all } \alpha \text{ with }\abs{\alpha}=\lfloor\theta\rfloor\}.
 \end{equation*}
 In particular if, for $\theta\in(0,1)$, we denote
 \begin{equation*}
  [u]_{H^\theta}\vcc=\biggl(\int_{\R^{n}\times \R^{n}}\frac{\abs{u(x)-u(y)}^{2}}{\abs{x-y}^{n+2\theta}}\dd{x}\dd{y}\biggr)^{1/2}
 \end{equation*}
 then $\norm{u}\vcc=\norm{u}_{H^{\lfloor\theta\rfloor}}+\sum_{\abs{\alpha}=\lfloor\theta\rfloor}[D^{\alpha}u]_{H^{\theta-\lfloor\theta\rfloor}}$ defines an equivalent norm on $H^{\theta}$.
\end{remark}
We also note the following scaling property of the fractional Laplacian:
\begin{equation}\label{eq:frac:Lap:scaling}
 \FL{\power/2}u(a \cdot)=a^{\power}(\FL{\power/2}u)(a\cdot)\qquad \text{for all }\power\in (0,2).
\end{equation}
Moreover, for later use, in analogy to the above notation we also introduce the space of bounded (Hölder) continuous functions by
\begin{equation*}
 C^{\theta}_{b}\vcc=\{u\colon \R^n\to\R^n\mid u\text{ is continuous and } \norm{u}_{C^{\theta}_{b}}<\infty\}.
\end{equation*}
Here, the norm $\norm{\cdot}_{C^{\theta}_{b}}$ is given by
\begin{equation*}
\norm{u}_{C^{\theta}_{b}}\vcc= \begin{cases}
  \sum_{\abs{\alpha}\leq \theta} \norm{D^{\alpha}u}_{L^{\infty}} & \text{if } \theta\in \N_0\\
  \sum_{\abs{\alpha}\leq \lfloor\theta\rfloor} \norm{D^{\alpha}u}_{L^{\infty}}+\sum_{\abs{\alpha}= \lfloor\theta\rfloor}\sup_{x\neq y}\frac{\abs{D^{\alpha}u(x)-D^{\alpha}u(y)}}{\abs{x-y}^{\theta-\lfloor\theta\rfloor}}&\text{if }\theta\not\in\N_0.
 \end{cases}
\end{equation*}

\subsection{Main result}

Let us consider the following approximation $\psi$ to $u$ solving~\eqref{eq:SH:qc}:
\begin{equation}\label{eq:first:approx}
 \eps\psi(x,t)=\eps\Bigl(A(\eps x,\eps^{2}t)\ee^{\im x}+\bar{A}(\eps x,\eps^{2}t)\ee^{-\im x}\Bigr)
\end{equation}
where $A=A(X,T)$ is a solution of the Ginzburg-Landau equation
\begin{equation}\label{eq:GL}
\del_{T}A=\power^{2}\del_{X}^{2}A+A-\Bigl(-\bigl(4+\frac{2}{\power^2+\mathfrak{c}^{+}}\bigr)a_1^2+3a_{2}\Bigr)\abs{A}^{2}A
\end{equation}
with $\mathfrak{c}^{+}$ as defined in \eqref{eq:c:pm} below. We note that the coefficient of the non-linearity is chosen in such a way that several terms of lower order in the residuum vanish as can be seen from the computation in Section~\ref{Sec:residuum}. 

Our main statement in this work is the following result stating that $\psi$ as defined in \eqref{eq:first:approx} approximates solutions to \eqref{eq:SH:qc} on a large time scale for $\eps$ sufficiently small.
\begin{theorem}\label{Thm:main}
 Let $\power\in[1,2)$, $\theta\geq 1$ and let $A\in C^1([0,T_*],H^{\theta+3})$ be a solution to \eqref{eq:GL}. Let $\psi$ be given by \eqref{eq:first:approx}. There exist solutions $u$ to \eqref{eq:SH:qc} such that
 \begin{equation*}
  \norm{u(\cdot, t)-\eps\psi(\cdot,t)}_{H^{\theta}}\lesssim \eps^{3/2} \qquad \text{for all } t\in [0,T_*/\eps^2].
 \end{equation*}
\end{theorem}

\subsection{Abstract approximation result}

To prove Theorem~\ref{Thm:main} we will rely on an abstract approximation result from \cite[Section 10.4.1]{ScU17} which we recall here for convenience. Precisely, the abstract equation
\begin{equation*}
 \del_{t}u=\Lambda u +N(u)
\end{equation*}
is considered for which a formal approximation $\eps \Psi$ is given on a time interval $[0,T_*/\eps^2]$ with $T_*>0$. Moreover, we denote by $R=\eps^{-\beta}(u-\eps\Psi)$ with $\beta\geq 3/2$ the corresponding scaled approximation error. Furthermore let $\Res(v)\vcc=-\del_{t}v+\Lambda v+N(v)$ be the \emph{residuum}.

The following abstract approximation result is contained in \cite[Thm. 10.4.3]{ScU17}.

\begin{theorem}\label{Thm:abstract:approximation}
Let $\X$ and $\Y$ be Banach spaces and let \emph{mode filters} $E_{c}$ and $E_{s}$ be given, i.e.\@ bounded linear projections on both $\X$ and $\Y$ (i.e.\@ $E_c^2=E_c$ and $E_s^2=E_s$) extracting the critical and stable modes respectively and such that they commute with the semi-group generated by $\Lambda$, i.e.\@
 \begin{equation*}
  \ee^{t\Lambda}E_{c}=E_{c}\ee^{t\Lambda}\qquad \ee^{t\Lambda}E_{s}=E_{s}\ee^{t\Lambda}.
 \end{equation*}
Moreover, assume the following:
\begin{enumerate}[label=(A\arabic*)]
 \item \label{It:ass:3} There exist constants $C_{\Lambda},\sigma_{c}\geq 0$, $\alpha\in[0,1)$ and $\sigma_{s}>0$ such that
 \begin{equation*}
  \begin{aligned}
   \norm{\ee^{t\Lambda}E_{c}}_{\Y\to\Y}&\leq C_{\Lambda}\ee^{\sigma_{c}\eps^{2}t} &\qquad \norm{\ee^{t\Lambda}E_{c}}_{\X\to\Y}&\leq C_{\Lambda}\ee^{\sigma_{c}\eps^{2}t}\\
   \norm{\ee^{t\Lambda}E_{s}}_{\Y\to\Y}&\leq C_{\Lambda}\ee^{-\sigma_{s}t}&\qquad
   \norm{\ee^{t\Lambda}E_{s}}_{\X\to\Y}&\leq C_{\Lambda}\max\{1,t^{-\alpha}\}\ee^{-\sigma_{s}t}.
  \end{aligned}
 \end{equation*}
  \item \label{It:ass:4} We have the estimates
 \begin{equation*}
  \begin{multlined}
   \norm{\eps^{-\beta}E_{c}(N(\eps \Psi +\eps^{\beta}R)-N(\eps \Psi))}_{\X}\\*
   \leq C_{1,c}\eps^{2}(\norm{R_{c}}_{\Y}+\norm{R_{s}}_{\Y})+C_{2,c}(M_{c},M_{s})\max\{\eps^{3},\eps^{\beta}\}(\norm{R_{c}}_{\Y}+\norm{R_{s}}_{\Y})^{2}\\*
   \shoveleft{\norm{\eps^{-(\beta+1)}E_{s}(N(\eps \Psi +\eps^{\beta}R)-N(\eps \Psi))}_{\X}}\\*
   \leq C_{1,s}\norm{R_{c}}_{\Y}+C_{1,s}\eps\norm{R_{s}}_{\Y}+C_{2,s}(M_{c},M_{s})\max\{\eps,\eps^{\beta-1}\}(\norm{R_{c}}_{\Y}+\norm{R_{s}}_{\Y})^{2}
  \end{multlined}
 \end{equation*}
 as long as $\norm{R_{c}}_{\Y}\leq M_{c}$ and $\norm{R_{s}}_{\Y}\leq M_{s}$ with constants $C_{1,c},C_{1,s}$ and $C_{2,c}$ and $C_{2,s}$ monotonically growing. Here $R=R_c+\eps R_s$ with $R_c$ and $R_s$ solving 
 \begin{equation*}
  \begin{split}
     \partial_t R_c&=\Lambda R_c +\eps^{-\beta} E_c\bigl(N(\eps \Psi+\eps^\beta R)-N(\eps \Psi)\bigr)+\eps^{-\beta}E_c \Res(\eps \Psi)\\
    \partial_t R_s&=\Lambda R_s +\eps^{-\beta-1} E_s\bigl(N(\eps \Psi+\eps^\beta R)-N(\eps \Psi)\bigr)+\eps^{-\beta-1}E_s \Res(\eps \Psi)
  \end{split}
 \end{equation*}
 in the mild sense. Moreover, for later use, let $\Psi_c=E_c(\eps \Psi)$ and $\Psi_s=\eps^{-1}E_s(\eps \Psi)$ be the stable and critical component of $\Psi$.
 \item \label{It:ass:5} The residuum $\Res(v)=-\del_{t}v+\Lambda v+N(v)$ satisfies
 \begin{equation}\label{eq:est:Res:abstract}
  \begin{aligned}
   \sup_{\tau\in[0,T_{0}/\eps^2]}\norm{E_{c}\Res(\eps \Psi(\tau))}_{\Y}\leq C_{res}\eps^{\beta+2}\quad \text{and}\quad \sup_{\tau\in[0,T_{0}/\eps^2]}\norm{E_{s}\Res(\eps \Psi(\tau))}_{\Y}\leq C_{res}\eps^{\beta+1}.
  \end{aligned}
 \end{equation}
\end{enumerate}
Under these assumptions there exist constants $C,\eps_{0}>0$ such that 
 \begin{equation*}
  \sup_{t\in[0,T_{0}/\eps^2]}\norm{u-\eps\Psi}_{\Y}\leq C\eps^{\beta}
 \end{equation*}
 for all $\eps\in(0,\eps_{0})$.
\end{theorem}

\section{Strategy and preparation}
\label{sec:strategy}

We want to apply Theorem~\ref{Thm:abstract:approximation} to the fractional Swift-Hohenberg equation~\eqref{eq:SH:qc} for which we have to choose appropriate mode filters and to verify the assumptions \cref{It:ass:3,It:ass:4,It:ass:5}. It seems convenient to work mainly in Fourier variables and thus it appears natural to choose Banach spaces of the form $H^\theta$. In fact, concerning the mode filters and \cref{It:ass:3} we can proceed as \cite[Section 10.4.2]{ScU17}. 

\subsection{The function spaces}

We will use $\X=\Y=H^{\theta}$ with $\theta\geq 0$. Moreover, we will fix the parameter $\beta=3/2$.

\begin{remark}\label{Rem:algebra}
 It is well-known (e.g.\@ \cite[Theorem 5.1]{BeH21}) that for $\theta>n/2$ the space $H^{\theta}(\R^n)$ is a Banach algebra, i.e.\@ for $u,v\in H^\theta(\R^n)$ we have $uv\in H^{\theta}(\R^n)$ and there exists a constant $C>0$ (depending only on $n$ and $\theta$) such that $\norm{uv}_{H^{\theta}}\leq C\norm{u}_{H^\theta}\norm{v}_{H^\theta}$.
\end{remark}

\subsection{The mode filters}\label{Sec:mode:filters}

For simplicity, we choose up to some constants, the same mode filters as in \cite[Section 10.4.2]{ScU17}, i.e.\@ $E_c$ is given through the Fourier symbol $m_{c}=\chi_{B_{\delta}(-1)\cup B_{\delta}(1)}$ where $\delta>0$ is sufficiently small, $\chi$ denotes the characteristic function and $B_{r}(a)$ the ball of radius $r$ around $a$. $E_{s}$ is then defined via the symbol $m_{s}=1-m_c$. Since the fractional Swift-Hohenberg operator is non-local with a non-smooth symbol, we use additionally a localization argument analogously as in \cite{Sch94b}, i.e.\@ we introduce another symbol $m_{0}=\chi_{B_{r_0}(0)}$ with $r_0>$ such that $3r_0<\delta$. The corresponding operator is denoted by $E_{0}$. To simplify the notation in several computations later in Section~\ref{Sec:residuum} we furthermore introduce the operator $E_{0}^{c}\vcc=\Id-E_0$. The continuity of these operators follows immediately (see \cite{ScU17}) and  $E_c$ and $E_s$ commute with the semi-group trivially (provided the latter exists -- see Proposition~\ref{Prop:semi-group} below) since everything is defined via Fourier multipliers.

 \subsection{Semi-group estimates}

  We first note that the fractional Swift-Hohenberg operator generates a semi-group on $H^\theta$.
  \begin{proposition}\label{Prop:semi-group}
  The operator $\Lambda u= -(1-\FL{\power/2})^{2}u+\eps^{2}u$ generates a semi-group on $H^{\theta}$ for all $\theta>0$ which is given via the Fourier symbol $\ee^{(-(1-\abs{\xi}^{\power})^{2}+\eps^{2})t}$.
 \end{proposition}
 \begin{proof}
It is direct to note that 
\begin{equation*}
\norm{\F^{-1}\ee^{(-(1-\abs{\xi}^{\power})^{2}+\eps^{2})t} \F u}_{H^{\theta}}=\norm{\F \F^{-1}\ee^{(-(1-\abs{\xi}^{\power})^{2}+\eps^{2})t} \F u}_{L_{\theta}^{2}}
\leq \sup_{\xi\in\R}\ee^{(-(1-\abs{\xi}^{\power})^{2}+\eps^{2})t}\norm{ \F u}_{L_{\theta}^{2}}
\end{equation*}
so we have a well-defined map on $H^{\theta}$. The further semi-group properties follow directly from the exponential formula of the Fourier symbol as usual.
 \end{proof}

 The next statement verifies assumption \ref{It:ass:3} of Theorem~\ref{Thm:abstract:approximation} while we also note that we will not exploit that the semi-group generated by $\Lambda$ is smoothing.
 \begin{proposition}
 Let $\delta\in(0,1)$ and $\power\in(0,2)$. There exist constants $C_{\Lambda}=C_{\Lambda}(\power,\delta)>0$ and $\sigma_{s}=\sigma_s(\power,\delta)>0$ such that
   \begin{equation*}
   \norm{\ee^{t\Lambda}E_{c}}_{H^{\theta}\to H^{\theta}}\leq C_{\Lambda}\ee^{\eps^{2}t}\qquad \text{and}\qquad 
   \norm{\ee^{t\Lambda}E_{s}}_{H^{\theta}\to H^{\theta}}\leq C_{\Lambda}\ee^{-\sigma_{s}t}.
 \end{equation*}
 \end{proposition}
 \begin{proof}
  The proof is essentially the same as in \cite{ScU17} but since it is short we recall it here for convenience. 
  \begin{multline*}
   \norm{\ee^{t\Lambda}E_{c}u}_{H^{\theta}}=\norm{\ee^{(-(1-\abs{\xi}^{\power})^{2}+\eps^{2})t}m_{c}(\xi)\widehat{u}(\xi)}_{L_{\theta}^{2}}\\*
   \leq \sup_{\xi\in\R}(\ee^{(-(1-\abs{\xi}^{\power})^{2}+\eps^{2})t}m_{c}(\xi))\norm{u}_{H^{\theta}}\leq \ee^{\eps^{2}t}\norm{u}_{H^{\theta}}.
  \end{multline*}
  To prove the second estimate we recall the definition of $m_s$ and define $\sigma_{s}(\power,\delta)\vcc=\frac{1}{2}(\min\{1,1-\abs{1-\delta}^\power,\abs{1+\delta}^\power-1\})^2>0$.  We then have
    \begin{multline*}
   \norm{\ee^{t\Lambda}E_{s}u}_{H^{\theta}}=\norm{\ee^{(-(1-\abs{\xi}^{\power})^{2}+\eps^{2})t}m_{s}(\xi)\widehat{u}(\xi)}_{L_{\theta}^{2}}\\*
   \leq \sup_{\xi\in\R}(\ee^{(-(1-\abs{\xi}^{\power})^{2}+\eps^{2})t}m_{s}(\xi))\norm{u}_{H^{\theta}}\leq \ee^{-\sigma_{s}(\power,\delta) t}\norm{u}_{H^{\theta}}
  \end{multline*}
  provided that $\eps^2\leq \sigma_{s}(\power,\delta)$.
 \end{proof}
 
 \subsection{The approximation}
 It turns out that instead of working directly with the approximation \eqref{eq:first:approx} it is more convenient to consider an improved approximation $\Psi$ involving higher order expressions (see also e.g.\@ \cite{ColletEckmann1,vanHarten,KSM92,Sch94b,ScU17}). More precisely, we will use
\begin{multline}\label{eq:improved:approx}
 \eps\Psi(x,t)=\eps\Bigl((E_{0}A)(\eps x,\eps^{2}t)\ee^{\im x}+(E_{0}\bar{A})(\eps x,\eps^{2}t)\ee^{-\im x}\Bigr)\\*
 +\eps^{2}\Bigl((E_{0}A_2)(\eps x,\eps^{2}t)\ee^{2\im x}+(E_{0}\bar{A}_2)(\eps x,\eps^{2}t)\ee^{-2\im x}+(E_{0}A_0)(\eps x,\eps^{2}t)\Bigr).
\end{multline}
Here $A_2$ and $A_0$ are given in terms of $A$ by
\begin{equation}\label{eq:choice:of:parameters}
   A_0=-2a_{1} \abs{A}^{2}\qquad \text{and}\qquad A_2=-\frac{a_{1}}{\power^{2}+\mathfrak{c}^{+}} A^2
\end{equation}
while
\begin{equation}\label{eq:c:pm}
 \mathfrak{c}^{\pm}=\int_{\pm 1}^{\pm 2}\bigl(3\del_{r}\abs{r}^{\power}\del_{r}^{2}\abs{r}^{\power}-(1-\abs{r}^{\power})\del_{r}^{3}\abs{r}^{\power}\bigr)(\pm 2-r)^{2}\dr.
\end{equation}
We note that by symmetry $\mathfrak{c}^{+}=\mathfrak{c}^{-}$. We also recall that $A$ solves \eqref{eq:GL}, i.e.\@
\begin{equation*}
\del_{T}A=\power^{2}\del_{X}^{2}A+A-\Bigl(-\bigl(4+\frac{2}{\power^2+\mathfrak{c}^{+}}\bigr)a_1^2+3a_{2}\Bigl)\abs{A}^{2}A. 
\end{equation*}
The choice of parameters is made to cancel several lower order terms in the residuum which will become clear later in Section~\ref{Sec:residuum}. We also emphasise that due to Lemma~\ref{Lem:cplus} we have $\power^2+\mathfrak{c}^{\pm}\neq 0$ for $\power>0$.
The following statement justifies that we can work with the improved approximation~\eqref{eq:improved:approx} instead of \eqref{eq:first:approx} (see also \cite{ScU17})
\begin{proposition}\label{Prop:improved:approx}
 Let $\theta\geq 1$. For $\psi$ as in \eqref{eq:first:approx} and $\Psi$ as in \eqref{eq:improved:approx} we have
\begin{equation*}
    \sup_{t\in[0,T_{0}/\eps^2]}\norm{\eps\psi-\eps\Psi}_{H^{\theta}}\leq C_{r_0,\power,a_1}\bigl(\norm{A}_{H^{\theta}}+\norm{A}_{H^{\theta}}^2\bigr)\eps^{3/2}.
\end{equation*}
\end{proposition}

\begin{proof}
 We note that 
 \begin{multline*}
     \eps\psi-\eps\Psi=\eps\Bigl((E_{0}^cA)(\eps x,\eps^{2}t)\ee^{\im x}+(E_{0}^c\bar{A})(\eps x,\eps^{2}t)\ee^{-\im x}\Bigr)\\*
 +\eps^{2}\Bigl((E_{0}A_2)(\eps x,\eps^{2}t)\ee^{2\im x}+(E_{0}\bar{A}_2)(\eps x,\eps^{2}t)\ee^{-2\im x}+(E_{0}A_0)(\eps x,\eps^{2}t)\Bigr).
 \end{multline*}
 Thus, by means of \cref{Lem:E0c,Rem:algebra,Lem:scaled:Sob:norm,eq:choice:of:parameters} we deduce
\begin{equation*}
  \sup_{t\in[0,T_{0}/\eps^2]}\norm{\eps\psi-\eps\Psi}_{H^{\theta}}\leq C_{r_0}\eps^{\theta+1/2}\norm{A}_{H^{\theta}}+C_{r_0,\power,a_1}\eps^{3/2}\norm{A}_{H^\theta}^2
\end{equation*}
from which the claim follows due to $\theta\geq 1$.
\end{proof}

\section{Estimating the non-linearities}
\label{sec:estimates}

In this section we estimate the non-linearities to verify Assumption~\ref{It:ass:4}. For this, we will split the approximation as well as the error in critical and stable parts following \cite{ScU17}. However as also pointed out there, the critical and stable part of the approximation, i.e.\@ $\Psi_c$ and $\Psi_s$ satisfy $\norm{\Psi_c}_{H^\theta},\norm{\Psi_s}_{H^{\theta}}=\Ord(\eps^{-1/2})$ which requires to estimate these expressions by Hölder norms. To do so, we will use the following two results.
\begin{proposition}\label{Prop:product:sobolev:hoelder}
 Let $\theta>0$ and $f\in H^{\theta}$.
 \begin{enumerate}[label=\roman*)]
  \item If $\theta\in \N_0$ then $\norm{fg}_{H^{\theta}}\lesssim \norm{f}_{H^{\theta}}\norm{g}_{C^{\theta}_b}$ for all $g\in C^{\theta}_{b}(\R^{n})$.
  \item If $\theta\not\in \N_{0}$ then $\norm{fg}_{H^{\theta}}\leq C_{\nu}\norm{f}_{H^{\theta}}\norm{g}_{C^{\theta+\nu}_b}$ for all $g\in C^{\theta+\nu}_{b}(\R^n)$ and all $0<\nu$.
 \end{enumerate}
\end{proposition}

\begin{proof}
 The first part follows immediately from the chain rule recalling also Remark~\ref{Rem:equivalence:Sobolev:norms}.
 
 To prove the second part, by means of the first part, it suffices to consider $\theta\in(0,1)$. Moreover, we can restrict to $0<\nu<1+\lfloor\theta\rfloor-\theta$. In this case, recalling Remark~\ref{Rem:equivalence:Sobolev:norms} it suffices to estimate $[fg]_{H^{\theta}}$ and we have
 \begin{multline*}
  [fg]_{H^{\theta}}=\biggl(\int_{\R^n\times\R^n}\frac{\abs{f(x)g(x)-f(y)g(y)}^2}{\abs{x-y}^{n+2\theta}}\dd{x}\dd{y}\biggr)^{1/2}\\*
  =\biggl(\int_{\R^n\times\R^n}\frac{\abs{(f(x)-f(y))g(x)+f(y)(g(x)-g(y))}^2}{\abs{x-y}^{n+2\theta}}\dd{x}\dd{y}\biggr)^{1/2}\\*
  \leq \sqrt{2}\biggl(\int_{\R^n\times\R^n}\frac{\abs{(f(x)-f(y)}^2}{\abs{x-y}^{n+2\theta}}\abs{g(x)}^2\dd{x}\dd{y}+\int_{\R^n\times\R^n}\frac{\abs{g(x)-g(y)}^2}{\abs{x-y}^{n+2\theta}}\abs{f(y)}^{2}\dd{x}\dd{y}\biggr)^{1/2}\\*
  \leq \sqrt{2}\biggl([f]_{H^{\theta}}\norm{g}_{C^0_b}^2+4\norm{g}_{C^{\theta+\nu}_b}^2\int_{\R^n}\abs{f(y)}^{2}\int_{\R^{n}}\frac{\min\{1,\abs{x-y}^{2\theta+2\nu}\}}{\abs{x-y}^{n+2\theta}}\dd{x}\dd{y}\biggr)^{1/2}\\*
  \leq \sqrt{2}\Bigl([f]_{H^{\theta}}\norm{g}_{C^0_b}^2+C_{\nu,\theta}\norm{f}_{L^2}^2\norm{g}_{C^{\theta+\nu}_b}^2\Bigr)^{1/2}.
 \end{multline*}
\end{proof}
The following result concerns a special case of Sobolev's Embedding Theorem ensuring that $\norm{\Psi_c}_{C^\theta_b},\norm{\Psi_s}_{C^{\theta}_b}=\Ord(1)$. It is a slight modification of \cite[Lemma~10.4.4]{ScU17} to where we refer for the corresponding proof.
\begin{lemma}\label{Lem:Psi:Hoelder}
 Let $A\in C([0,T_*],H^{\theta_A})$ be a solution to \eqref{eq:GL} with $\theta_A>\theta+\frac{1}{2}$. Let $\Psi$ be given by \eqref{eq:improved:approx} and $\Psi_c=E_c(\Psi)$ as well as $\Psi_s=E_s(\Psi)$. For sufficiently small $\eps>0$ we have
\begin{equation*}
 \sup_{t\in[0,T_*/\eps^2]}\Bigl(\norm{\Psi_c}_{C^{\theta}_b}+\norm{\Psi_s}_{C^{\theta}_b}\Bigr)\lesssim 1.
\end{equation*}
\end{lemma}

After this preparation, we will now estimate the non-linearities to justify \ref{It:ass:4}. For the quadratic terms we obtain
\begin{multline*}
 (\eps \Psi +\eps^{\beta}R)^{2}-(\eps\Psi)^{2}\\*
 \shoveleft{=(\eps \Psi_c+\eps^2 \Psi_s +\eps^{\beta} R_c +\eps^{\beta+1} R_{s})^{2}-(\eps \Psi_c +\eps^2 \Psi_s)^{2}}\\*
 \shoveleft{=\eps^{2\beta}R_{c}^{2}+\eps^{2(\beta+1)}R_{s}^{2}+2\Bigl(\eps^{1+\beta}\Psi_{c}R_{c}+\eps^{\beta+2}\Psi_{c}R_{s}+\eps^{2+\beta}\Psi_{s}R_{c}+\eps^{3+\beta}\Psi_{s}R_{s}+\eps^{2\beta+1}R_{c}R_{s}}\Bigr).
\end{multline*}
Exploiting the fact that $E_{c}(f_c g_c)=0$ and Proposition~\ref{Prop:product:sobolev:hoelder} we obtain together with Remark~\ref{Rem:algebra} that
\begin{multline}\label{non-lin:1}
 \eps^{-\beta}\norm*{E_{c}\bigl((\eps \Psi +\eps^{\beta}R)^{2}-(\eps\Psi)^{2}\bigr)}_{\X}\\*
 \lesssim \eps^{2}\norm{\Psi_{c}}_{C_{b}^{\theta}}\norm{R_{s}}_{\Y}+\eps^{2}\norm{\Psi_{s}}_{C_{b}^{\theta}}\norm{R_{c}}_{\Y}+\eps^{3}\norm{\Psi_{s}}_{C_{b}^{\theta}}\norm{R_{s}}_{\Y}+\eps^{2+\beta}\norm{R_{s}}_{\Y}^{2}+\eps^{\beta+1}\norm{R_{c}}_{\Y}\norm{R_{s}}_{\Y}\\*
\lesssim \eps^{2}\bigl(\norm{\Psi_{c}}_{C_{b}^{\theta}}+\norm{\Psi_{s}}_{C_{b}^{\theta}}\bigr)\bigl(\norm{R_{s}}_{\Y}+\norm{R_{c}}_{\Y}\bigr)+\eps^{1+\beta}\bigl(\norm{R_{s}}_{\Y}+\norm{R_{c}}_{\Y}\bigr)^2
\end{multline}
as well as
\begin{multline}\label{non-lin:2}
 \eps^{-\beta-1}\norm*{E_{s}\bigl((\eps\Psi+\eps^{\beta}R)^{2}-(\eps\Psi)^{2}\bigr)}_{\X}\\*
 \shoveleft{\lesssim \bigl(\norm{\Psi_c}_{C_{b}^{\theta}}+\eps\norm{\Psi_{s}}_{C_{b}^{\theta}}\bigr)\norm{R_{c}}_{\Y}+\bigl(\eps\norm{\Psi_{c}}_{C_{b}^{\theta}}+\eps^{2}\norm{\Psi_{s}}_{C_{b}^{\theta}}\bigr)\norm{R_{s}}_{\Y}}\\*
 \shoveright{+\eps^{\beta-1}\norm{R_{c}}_{\Y}^{2}+\eps^{\beta-1}\norm{R_{s}}_{\Y}^{2}+\eps^{\beta}\norm{R_{c}}_{\Y}\norm{R_{s}}_{\Y}}\\*
\lesssim \bigl(\norm{\Psi_c}_{C_{b}^{\theta}}+\eps\norm{\Psi_{s}}_{C_{b}^{\theta}}\bigr)\bigl(\norm{R_{c}}_{\Y}+\eps\norm{R_{s}}_{\Y}\bigr)+\eps^{\beta-1}\bigl(\norm{R_{c}}_{\Y}+\norm{R_{s}}_{\Y}\bigr)^{2}.
\end{multline}
In the same way, we can proceed for the cubic terms to get
\begin{multline*}
 (\eps\Psi+\eps^{\beta}R)^{3}-(\eps\Psi)^{3}=(\eps\Psi_c+\eps^{2}\Psi_s+\eps^{\beta}R_{c}+\eps^{\beta+1}R_{s})^{3}-(\eps\Psi_{c}+\eps^{2}\Psi_{s})^{3}\\*
 =3\Bigl(\eps^{2+\beta}\Psi_{c}^{2}R_{c}+\eps^{3+\beta}\Psi_{c}^{2}R_{s}+\eps{4+\beta}\Psi_{s}^{2}R_{c}+\eps^{5+\beta}\Psi_{s}^{2}R_{s}+\eps^{1+2\beta}\Psi_{c}R_{c}^{2}+\eps^{2+2\beta}\Psi_{s}R_{c}^{2}\\*
 +\eps^{1+3\beta}R_{c}^{2}R_{s}+\eps^{3+2\beta}\Psi_{c}R_{s}^{2}+\eps^{4+2\beta}\Psi_{s}R_{s}^{2}+\eps^{2+3\beta}R_{c}R_{s}^{2}\Bigr)+\eps^{3\beta}R_{c}^{3}+\eps^{3\beta+3}R_{s}^{3}.
\end{multline*}
Proceeding as above we get
\begin{multline*}
 \eps^{-\beta}\norm*{E_{c}\bigl((\eps\Psi+\eps^{\beta}R)^{3}-(\eps\Psi)^{3}\bigr)}_{\X}\\*
 \lesssim\eps^{2}\norm{\Psi_{c}}_{C_{b}^{\theta}}^{2}\norm{R_{c}}_{\Y}+\eps^{3}\norm{\Psi_{c}}_{C_{b}^{\theta}}^{2}\norm{R_{s}}_{\Y}+\eps^{4}\norm{\Psi_{s}}_{C_{b}^{\theta}}^{2}\norm{R_{c}}_{\Y}+\eps^{5}\norm{\Psi_{s}}_{C_{b}^{\theta}}^{2}\norm{R_{s}}_{\Y}\\*
 +\eps^{1+\beta}\norm{\Psi_{c}}_{C_{b}^{\theta}}\norm{R_{c}}_{\Y}^{2}+\eps^{2+\beta}\norm{\Psi_{s}}_{C_{b}^{\theta}}\norm{R_{c}}_{\Y}^{2}+\eps^{1+2\beta}\norm{R_{c}}_{\Y}^{2}\norm{R_{s}}_{\Y}\\*
 +\eps^{3+\beta}\norm{\Psi_{c}}_{C_{b}^{\theta}}\norm{R_{s}}^{2}_{\Y}+\eps^{4+\beta}\norm{\Psi_{s}}_{C_{b}^{\theta}}\norm{R_{s}}_{\Y}^{2}+\eps^{2+2\beta}\norm{R_{c}}_{\Y}\norm{R_{s}}_{\Y}^{2}\\*
 +\eps^{2\beta}\norm{R_{c}}_{\Y}^{3}+\eps^{3+2\beta}\norm{R_{s}}_{\Y}^{3}.
\end{multline*}
Estimating the leading order and combining we find
\begin{multline}\label{non-lin:3}
 \eps^{-\beta}\norm*{E_{c}\bigl((\eps\Psi+\eps^{\beta}R)^{3}-(\eps\Psi)^{3}\bigr)}_{\X}\\*
 \lesssim\eps^{2}\bigl(\norm{\Psi_{c}}_{C_{b}^{\theta}}^{2}+\norm{\Psi_{s}}_{C_{b}^{\theta}}^{2}\bigr)\bigl(\norm{R_{c}}_{\Y}+\norm{R_{s}}_{\Y}\bigr)+\eps^{1+\beta}\bigl(\norm{\Psi_{c}}_{C_{b}^{\theta}}+\norm{\Psi_{s}}_{C_{b}^{\theta}}\bigr)\bigl(\norm{R_{c}}_{\Y}+\norm{R_{s}}_{\Y}\bigr)^{2}\\*
 +\eps^{2\beta}\bigl(\norm{R_{c}}_{\Y}+\norm{R_{s}}_{\Y}\bigr)\bigl(\norm{R_{c}}_{\Y}+\norm{R_{s}}_{\Y}\bigr)^{2}.
\end{multline}
Analogously, we have
\begin{multline*}
 \eps^{-\beta-1}\norm*{E_{s}\bigl((\eps\Psi+\eps^{\beta}R)^{3}-(\eps\Psi)^{3}\bigr)}_{\X}\\*
 \lesssim\eps\norm{\Psi_{c}}_{C_{b}^{\theta}}^{2}\norm{R_{c}}_{\Y}+\eps^{2}\norm{\Psi_{c}}_{C_{b}^{\theta}}^{2}\norm{R_{s}}_{\Y}+\eps^{3}\norm{\Psi_{s}}_{C_{b}^{\theta}}^{2}\norm{R_{c}}_{\Y}+\eps^{4}\norm{\Psi_{s}}_{C_{b}^{\theta}}^{2}\norm{R_{s}}_{\Y}\\*
 +\eps^{\beta}\norm{\Psi_{c}}_{C_{b}^{\theta}}\norm{R_{c}}_{\Y}^{2}+\eps^{1+\beta}\norm{\Psi_{s}}_{C_{b}^{\theta}}\norm{R_{c}}_{\Y}^{2}+\eps^{2\beta}\norm{R_{c}}_{\Y}^{2}\norm{R_{s}}_{\Y}\\*
 +\eps^{2+\beta}\norm{\Psi_{c}}_{C_{b}^{\theta}}\norm{R_{s}}^{2}_{\Y}+\eps^{3+\beta}\norm{\Psi_{s}}_{C_{b}^{\theta}}\norm{R_{s}}_{\Y}^{2}+\eps^{1+2\beta}\norm{R_{c}}_{\Y}\norm{R_{s}}_{\Y}^{2}\\*
 +\eps^{2\beta-1}\norm{R_{c}}_{\Y}^{3}+\eps^{2+2\beta}\norm{R_{s}}_{\Y}^{3}.
\end{multline*}
This yields
\begin{multline}\label{non-lin:4}
 \eps^{-\beta-1}\norm*{E_{s}\bigl((\eps\Psi+\eps^{\beta}R)^{3}-(\eps\Psi)^{3}\bigr)}_{\X}\\*
 \lesssim\eps\bigl(\norm{\Psi_{c}}_{C_{b}^{\theta}}^{2}+\eps\norm{\Psi_{s}}_{C_{b}^{\theta}}^{2}\bigr)\bigl(\norm{R_{c}}_{\Y}+\eps\norm{R_{s}}_{\Y}\bigr)+\eps^{\beta}\bigl(\norm{\Psi_{c}}_{C_{b}^{\theta}}+\norm{\Psi_{s}}_{C_{b}^{\theta}}\bigr)\bigl(\norm{R_{c}}_{\Y}+\norm{R_{s}}_{\Y}\bigr)^{2}\\*
+\eps^{2\beta-1}\bigl(\norm{R_{c}}_{\Y}+\norm{R_{s}}_{\Y}\bigr)\bigl(\norm{R_{c}}_{\Y}+\norm{R_{s}}_{\Y}\bigr)^{2}.
\end{multline}
Summarising \cref{non-lin:1,non-lin:2,non-lin:3,non-lin:4} we have for $\norm{R_{c}}_{\Y}\leq M_c$ and $\norm{R_{s}}_{\Y}\leq M_s$ that
\begin{multline*}
  \norm{\eps^{-\beta}E_{c}(N(\eps \Psi +\eps^{\beta}R)-N(\eps \Psi))}_{\X}\\*
   \lesssim \eps^{2}\bigl(\norm{\Psi_{c}}_{C_{b}^{\theta}}+\norm{\Psi_{s}}_{C_{b}^{\theta}}+\norm{\Psi_{c}}_{C_{b}^{\theta}}^{2}+\norm{\Psi_{s}}_{C_{b}^{\theta}}^{2}\bigr)\bigl(\norm{R_{s}}_{\Y}+\norm{R_{c}}_{\Y}\bigr)\\*
+\min\bigl\{\eps^{1+\beta},\eps^{2\beta}\bigr\}\bigl(1+\norm{\Psi_{c}}_{C_{b}^{\theta}}+\norm{\Psi_{s}}_{C_{b}^{\theta}}+M_s+M_c\bigr)\bigl(\norm{R_{c}}_{\Y}+\norm{R_{s}}_{\Y}\bigr)^{2} 
\end{multline*}
and
\begin{multline*}
   \norm{\eps^{-(\beta+1)}E_{s}(N(\eps \Psi +\eps^{\beta}R)-N(\eps \Psi))}_{\X}\\*
   \lesssim \bigl(\norm{\Psi_c}_{C_{b}^{\theta}}+\eps\norm{\Psi_{s}}_{C_{b}^{\theta}}+\eps\norm{\Psi_{c}}_{C_{b}^{\theta}}^{2}+\eps^2\norm{\Psi_{s}}_{C_{b}^{\theta}}^{2}\bigr)\bigl(\norm{R_{c}}_{\Y}+\eps\norm{R_{s}}_{\Y}\bigr)+\\*
+\eps^{\beta-1}\bigl(1+\norm{\Psi_{c}}_{C_{b}^{\theta}}+\norm{\Psi_{s}}_{C_{b}^{\theta}}+M_c+M_s\bigr)\bigl(\norm{R_{c}}_{\Y}+\norm{R_{s}}_{\Y}\bigr)^{2}.
\end{multline*}
Together with Lemma~\ref{Lem:Psi:Hoelder} this justifies \ref{It:ass:4} of Theorem~\ref{Thm:abstract:approximation}.
   
\section{The fractional Swift-Hohenberg operator}\label{Sec:fract:SW}

\subsection{Approximation of the fractional Swift-Hohenberg operator}\label{Sec:approximation:SH-operator}

Due to the non-locality of the fractional Laplacian, we cannot directly compute the residuum. Instead some technical preparation is required which is the content of this section. More precisely, we exploit that the Fourier symbol $-(1-\abs{\xi}^\power)^2$ of $-(1-(-\Delta)^{\power/2})^2$ is smooth except for $\xi=0$ which allows to Taylor approximate around the critical modes $\xi=\pm 1$ analogously to the classical situation $\power=2$ while the case $\xi=0$ will be treated separately.
\begin{lemma}\label{Lem:Taylor:symbol}
 For $\xi\in (-\infty, 0)$ and $\xi\in (0,\infty)$ respectively, we have the following representation of the Fourier symbol $-(1-\abs{\xi}^{\power})^{2}$ of $-(1-(-\Delta)^{\power/2})^2$:
 \begin{equation*}
  -(1-\abs{\xi}^{\power})^{2}=-\power^{2}(\xi\mp 1)^{2}-\int_{\pm 1}^{\xi}\bigl(3\del_{r}\abs{r}^{\power}\del_{r}^{2}\abs{r}^{\power}-(1-\abs{r}^{\power})\del_{r}^{3}\abs{r}^{\power}\bigr)(\xi-r)^{2}\dr
 \end{equation*}
\end{lemma}
\begin{proof}
 This follows immediately by Taylor approximation using that $\xi\mapsto \abs{\xi}^{\power}$ is $C^{\infty}$ apart from $0$.
\end{proof}
Moreover, for modes centered at $\pm 2$ we will need another representation of the remainder.
\begin{lemma}\label{Lem:remainder:at:two}
 For $\xi\in (-\infty, 0)$ and $\xi\in (0,\infty)$ respectively, we have
  \begin{multline*}
  \int_{\pm 1}^{\xi}\bigl(3\del_{r}\abs{r}^{\power}\del_{r}^{2}\abs{r}^{\power}-(1-\abs{r}^{\power})\del_{r}^{3}\abs{r}^{\power}\bigr)(\xi-r)^{2}\dr\\*
  \shoveleft{=\int_{\pm 1}^{\pm 2}\bigl(3\del_{r}\abs{r}^{\power}\del_{r}^{2}\abs{r}^{\power}-(1-\abs{r}^{\power})\del_{r}^{3}\abs{r}^{\power}\bigr)(\pm 2-r)^{2}\dr}\\*
  +\int_{\pm 1}^{\pm 2}\bigl(3\del_{r}\abs{r}^{\power}\del_{r}^{2}\abs{r}^{\power}-(1-\abs{r}^{\power})\del_{r}^{3}\abs{r}^{\power}\bigr)(\xi\pm 2 -2r)\dr(\xi\mp 2)\\*
  +\int_{\pm 2}^{\xi}\bigl(3\del_{r}\abs{r}^{\power}\del_{r}^{2}\abs{r}^{\power}-(1-\abs{r}^{\power})\del_{r}^{3}\abs{r}^{\power}\bigr)(\xi-r)^{2}\dr.
 \end{multline*}
\end{lemma}
\begin{proof}
 This follows immediately noting that 
 \begin{equation*}
  (\xi-r)^{2}-(\pm 2-r)^2-(\xi\pm 2-2r)(\xi\mp 2)=0.
 \end{equation*}
\end{proof}
For later use, we introduce the following notation:
\begin{equation}\label{eq:mult:r:pm}
 \mathfrak{r}^{\pm}(\xi)=\int_{\pm 1}^{\xi}\bigl(3\del_{r}\abs{r}^{\power}\del_{r}^{2}\abs{r}^{\power}-(1-\abs{r}^{\power})\del_{r}^{3}\abs{r}^{\power}\bigr)(\xi-r)^{2}\dr
\end{equation}
and
\begin{equation}\label{eq:mult:m:pm}
 \begin{aligned}
  m^{1,\pm}(\xi)&=\int_{\pm 1}^{\pm 2}\bigl(3\del_{r}\abs{r}^{\power}\del_{r}^{2}\abs{r}^{\power}-(1-\abs{r}^{\power})\del_{r}^{3}\abs{r}^{\power}\bigr)(\xi\pm 2 -2r)\dr(\xi\mp 2)\\
  m^{2,\pm}(\xi)&=\int_{\pm 2}^{\xi}\bigl(3\del_{r}\abs{r}^{\power}\del_{r}^{2}\abs{r}^{\power}-(1-\abs{r}^{\power})\del_{r}^{3}\abs{r}^{\power}\bigr)(\xi-r)^{2}\dr.
 \end{aligned}
\end{equation}
We denote the operators corresponding to the Fourier multipliers $m^{1,\pm}$ and $m^{2,\pm}$ respectively by $M^{1,\pm}$ and $M^{2,\pm}$. Let us also recall from~\eqref{eq:c:pm} that
\begin{equation*}
 \mathfrak{c}^{\pm}=\int_{\pm 1}^{\pm 2}\bigl(3\del_{r}\abs{r}^{\power}\del_{r}^{2}\abs{r}^{\power}-(1-\abs{r}^{\power})\del_{r}^{3}\abs{r}^{\power}\bigr)(\pm 2-r)^{2}\dr.
\end{equation*}
The following lemma shows that the function $A_2$ in \eqref{eq:choice:of:parameters} is well-defined.

\begin{lemma}\label{Lem:cplus}
 For $\power\in (0,2)$ we have
 \begin{equation*}
  \power^{2}+\mathfrak{c}^{\pm}>0.
 \end{equation*}
\end{lemma}

\begin{proof}
 An explicit evaluation of the integral yields
 \begin{equation*}
  \mathfrak{c}^{+}=\int_{1}^{2}\Bigl(3\del_{r}r^{\power}\del_{r}^{2}r^{\power}-(1-r^{\power})\del_{r}^{3}r^{\power}\Bigr)(2-r)^{2}\dd{r}=2^{2\power}-2^{\power+1}+(1-\power^2)
 \end{equation*}
 from which the claim immediately follows.
\end{proof}
From the explicit formulas for $\mathfrak{r}^{\pm}$, $m^{1,\pm}$ and $m^{2,\pm}$ we get the following estimates.
\begin{lemma}\label{Lem:est:symbols}
 For any $\nu\leq 1/2$ there exists a constant $C_{\nu}$ such that 
 \begin{equation*}
  \begin{aligned}
   \abs{\mathfrak{r}^{\pm}(\xi)}&\leq C_\nu \abs{\xi\mp 1}^{3} && \text{if } \xi\in [\pm 1-\nu,\pm 1+\nu]\\
   \abs{m^{1,\pm}(\xi)}&\leq C_\nu \abs{\xi\mp 2} && \text{if } \xi\in [\pm 2-\nu,\pm 2+\nu]\\
   \abs{m^{2,\pm}(\xi)}&\leq C_\nu \abs{\xi\mp 2}^{3} && \text{if } \xi\in [\pm 2-\nu,\pm 2+\nu].
  \end{aligned}
 \end{equation*}
\end{lemma}

\begin{proof}
 This is a straightforward estimate exploiting that $r\mapsto \bigl(3\del_{r}\abs{r}^{\power}\del_{r}^{2}\abs{r}^{\power}-(1-\abs{r}^{\power})\del_{r}^{3}\abs{r}^{\power}\bigr)$ is smooth apart from $r=0$.
\end{proof}

\subsection{Scaling properties}

In this subsection we collect several elementary estimates for the interplay between Fourier multipliers, scaled variables and the $H^\theta$ norm. These results will particularly be useful when estimating the residuum in Section~\ref{Sec:residuum}.

The following lemma provides $H^\mu$ estimates on products.
\begin{lemma}\label{Lem:Sob:norm:prod}
 Let $f,g,h\in H^{\mu}$ with $\mu>1/2$. We have
 \begin{equation*}
  \begin{aligned}
   \norm{fg}_{H^{\mu}}&\leq \norm{f}_{H^{\mu}}\norm{\hat{g}}_{L^{1}}+\norm{\hat{f}}_{L^{1}}\norm{g}_{H^{\mu}}\\
   \norm{fgh}_{H^{\mu}}&\leq \norm{f}_{H^{\mu}}\norm{\hat{g}}_{L^{1}}\norm{\hat{h}}_{L^{1}}+\norm{\hat{f}}_{L^{1}}\norm{g}_{H^{\mu}}\norm{\hat{h}}_{L^{1}}+\norm{\hat{f}}_{L^{1}}\norm{\hat{g}}_{L^{1}}\norm{h}_{H^{\mu}}.
  \end{aligned}
 \end{equation*}
\end{lemma}

\begin{proof}
 The first inequality is shown in \cite[Lemma~7.3.29]{ScU17}, the second one follows immediately from the first since $\norm{\hat{u}\ast \hat{v}}_{L^{1}}\leq \norm{\hat{u}}_{L^{1}}\norm{\hat{v}}_{L^{1}}$.
\end{proof}

The next statement shows that $E_0$ is a bounded linear operator.

\begin{lemma}\label{Lem:E0}
 Let $f\in H^{\mu}$. We have 
 \begin{equation*}
  \norm{E_{0}f}_{H^{\mu}}\leq \norm{f}_{H^{\mu}}.
 \end{equation*}
 Moreover, if $\mu>1/2$ we have
 \begin{equation*}
  \norm{\widehat{E_{0}f(\eps \cdot)}}_{L^{1}}\lesssim \norm{\hat{f}}_{L^{1}}\lesssim \norm{f}_{H^{\mu}}.
 \end{equation*}
\end{lemma}

\begin{proof}
 The first claim is an immediate consequence of $m_0\leq 1$. The first estimate in the second claim follows again from the definition of $E_0$ and $m_0\leq 1$. The second estimate is a direct consequence of Hölder's inequality (see also Lemma~\ref{Lem:L1:Htheta}).
\end{proof}

\begin{remark}
    The estimate in Lemma~\ref{Lem:E0} is far from being optimal since $E_{0}$ is smoothing. In fact we can easily obtain $\norm{f}_{L^{2}}$ on the right-hand side. However, we will not use such an improved estimate in the sequel.
\end{remark}

The operator $E_0^c=\Id-E_0$ provides an additional order of $\eps^\mu$ in the $H^\mu$-norm when applied to scaled variables as can be seen from the next statement (see also \cite{Sch94b}).

\begin{lemma}\label{Lem:E0c}
 For $f\in H^{\mu}$ we have
 \begin{equation*}
  \norm{(E_{0}^{c}f(\eps \cdot))\ee^{\im k \cdot}}_{H^{\mu}}\leq C_{r_0}(1+k+k^{2})^{\mu/2}\eps^{\mu-\frac{1}{2}}\norm{f}_{H^{\mu}}.
 \end{equation*}
 Moreover, if $f\in H^{\mu+\gamma}$ with $\gamma>1/2$ we have
  \begin{equation*}
  \norm{\widehat{(E_{0}^{c}f(\eps \cdot))\ee^{\im k\cdot}}}_{L^{1}}\leq C_{r_0,\gamma} \eps^{\mu}\norm{f}_{H^{\mu+\gamma}}.
 \end{equation*}
\end{lemma}

\begin{proof}
To prove the first claim we note that
 \begin{multline*}
  \norm{(E_{0}^{c}f(\eps \cdot))\ee^{\im k \cdot}}_{H^{\mu}}=\norm{\widehat{E_{0}^{c}f(\eps \cdot))\ee^{\im k \cdot}}}_{L^{2}_{\mu}}=\frac{1}{\eps}\biggl(\int_{\R}\abs[\Big]{(1-m_0(\xi-k))\hat{f}\Bigl(\frac{\xi-k}{\eps}\Bigr)(1+\xi^{2})^{\mu/2}}^{2}\dd{\xi}\biggr)^{1/2}\\*
  =\frac{1}{\eps}\biggl(\int_{\R}\abs[\Big]{(1-m_0(\xi))\hat{f}\Bigl(\frac{\xi}{\eps}\Bigr)(1+(k+\xi)^{2})^{\mu/2}}^{2}\dd{\xi}\biggr)^{1/2}\\*
  \leq (1+k+k^{2})^{\mu/2}\frac{1}{\eps}\biggl(\int_{\R}\abs[\Big]{(1-m_0(\xi))\hat{f}\Bigl(\frac{\xi}{\eps}\Bigr)(1+\xi^{2})^{\mu/2}}^{2}\dd{\xi}\biggr)^{1/2}.
 \end{multline*}
Due to the assumption on $m_0$ we have $(1-m_0(\xi))(1+\xi^{2})^{\mu/2}\leq \abs{\xi}^{\mu}(1+1/r_0^{2})$. For $\abs{\xi}\geq r_0$ this follows from $1\leq \abs{\xi/r_0}$ while for $\abs{\xi}\leq r_0$ we even have $(1-m_0(\xi))(1+\xi^{2})^{\mu/2}=0$. Thus, we can further estimate
 \begin{multline*}
  \norm{(E_{0}^{c}f(\eps \cdot))\ee^{\im k \cdot}}_{H^{\mu}}\leq (1+k+k^{2})^{\mu/2}\Bigl(1+\frac{1}{r_0^{2}}\Bigr)\frac{1}{\eps}\biggl(\int_{\R}\abs[\Big]{\hat{f}\Bigl(\frac{\xi}{\eps}\Bigr)\abs{\xi}^{\mu}}^{2}\dd{\xi}\biggr)^{1/2}\\*
  =(1+k+k^{2})^{\mu/2}\Bigl(1+\frac{1}{r_0^{2}}\Bigr)\eps^{\mu-\frac{1}{2}}\biggl(\int_{\R}\abs*{\hat{f}(\xi)\abs{\xi}^{\mu}}^{2}\dd{\xi}\biggr)^{1/2}\leq (1+k+k^{2})^{\mu/2}\Bigl(1+\frac{1}{r_0^{2}}\Bigr)\eps^{\mu-\frac{1}{2}}\norm{f}_{H^{\mu}}.
 \end{multline*}
 The second claim follow similarly noting that 
  \begin{equation*}
  \norm{\widehat{(E_{0}^{c}f(\eps \cdot))\ee^{\im k\cdot}}}_{L^{1}}=\frac{1}{\eps}\int_{\R}(1-m_0(\xi))\abs[\Big]{\hat{f}\Bigl(\frac{\xi}{\eps}\Bigr)}\dd{\xi}.
 \end{equation*}
 Due to the properties of $m_0$ we have $(1-m_0(\xi))\leq C_{r_0} \abs{\xi}^{\mu}$ for all $\mu\geq 0$. Thus,
  \begin{multline*}
    \norm{\widehat{(E_{0}^{c}f(\eps \cdot))\ee^{\im k\cdot}}}_{L^{1}}\leq C_{r_0}\frac{1}{\eps}\int_{\R}\abs{\xi}^{\mu}\abs[\Big]{\hat{f}\Bigl(\frac{\xi}{\eps}\Bigr)}\dd{\xi}=C_{r_0}\eps^{\mu}\int_{\R}\abs{\xi}^{\mu}\abs{\hat{f}(\xi)}\dd{\xi}\\*
    =C_{r_0}\eps^{\mu}\int_{\R}(1+\xi^{2})^{-\gamma/2}(1+\xi^{2})^{\gamma/2}\abs{\xi}^{\mu}\abs{\hat{f}(\xi)}\dd{\xi}\leq C_{r_0}\eps^{\mu}\biggl(\int_{\R}(1+\xi^{2})^{-\gamma}\dd{\xi}\biggr)^{1/2}\norm{f}_{H^{\mu+\gamma}}\\*
    \leq C_{r_0,\gamma}\eps^{\mu}\norm{f}_{H^{\mu+1}}.
  \end{multline*}
\end{proof}
The following two lemmas provide estimates of scaled variables in $H^\mu$ and $L^1$ respectively.
\begin{lemma}\label{Lem:scaled:Sob:norm}
 For $\eps\leq 1$ and $f\in H^{\mu}$, we have
 \begin{equation*}
  \norm{f(\eps \cdot)\ee^{\im k\cdot}}_{H^{\mu}}\leq (1+k+k^{2})^{\mu/2}\eps^{-1/2}\norm{f}_{H^{\mu}}.
 \end{equation*}
\end{lemma}

\begin{proof}
 \begin{multline*}
  \norm{f(\eps \cdot)\ee^{\im k\cdot}}_{H^{\mu}}=\norm{\widehat{f(\eps \cdot)\ee^{\im k\cdot}}}_{L^{2}_{\mu}}=\frac{1}{\eps}\biggl(\int_{\R}\abs*{\hat{f}\Bigl(\frac{\xi-k}{\eps}\Bigr)(1+\xi^{2})^{\mu/2}}^{2}\dd{\xi}\biggr)^{1/2}\\*
  =\eps^{-1/2}\biggl(\int_{\R}\abs*{\hat{f}(\xi)(1+(k+\eps \xi)^{2})^{\mu/2}}^{2}\dd{\xi}\biggr)^{1/2}\leq (1+k+k^{2})^{\mu/2}\eps^{-1/2}\norm{f}_{H^{\mu}}.
 \end{multline*}
Note that in the last step we used that $1+(k+\eps \xi)^{2}=1+k^2+2k\eps \xi+\eps^{2}\xi^{2}\leq (1+k+k^{2})+(1+k)\xi^{2}$.
\end{proof}

\begin{lemma}\label{Lem:L1:Htheta}
 For any $k\in \R$ and $f$ such that $\hat{f}\in L^1$ we have
 \begin{equation*}
  \norm{\widehat{f(\eps \cdot)\ee^{\im k \cdot}}}_{L^{1}}=\norm{\hat{f}}_{L^{1}}.
 \end{equation*}
Moreover, if $\mu>1/2$ and $f\in H^{\mu}(\R)$ we have
 \begin{equation*}
  \norm{\widehat{f(\eps \cdot)\ee^{\im k \cdot}}}_{L^{1}}\leq C_{\mu}\norm{f}_{H^{\mu}}.
 \end{equation*}
\end{lemma}

\begin{proof}
 \begin{equation*}
  \norm{\widehat{f(\eps \cdot)\ee^{\im k \cdot}}}_{L^{1}}=\frac{1}{\eps}\int_{\R}\abs[\Big]{\hat{f}\Bigl(\frac{\xi-k}{\eps}\Bigr)}\dd{\xi}=\int_{\R}\abs{\hat{f}(\xi)}\dd{\xi}=\norm{\hat{f}}_{L^{1}}.
 \end{equation*}
Moreover,
 \begin{equation*}
  \norm{\hat{f}}_{L^{1}}=\int_{\R}\abs{\hat{f}(\xi)}\dd{\xi}=\int_{\R}(1+\xi^{2})^{-\mu/2}(1+\xi^{2})^{\mu/2}\abs{\hat{f}(\xi)}\dd{\xi}\leq \biggl((1+\xi^{2})^{-\mu}\biggr)^{1/2}\norm{f}_{H^{\mu}}.
 \end{equation*}
\end{proof}

The next statement provides the scaling properties and estimates in $H^\mu$ of the fractional Laplacian applied to scaled spatial variables.

\begin{lemma}\label{Lem:frac:Lap:scaling}
 Let $\eps\leq 1$ and $\nu\in(0,1)$. For $f\in H^{\mu+2\nu}$, we have
 \begin{equation*}
  (-\Delta)^{\nu}(E_{0}f(\eps \cdot))=\eps^{2\nu}E_{0}((-\Delta)^{\nu}f)(\eps \cdot)
 \end{equation*}
and
 \begin{equation*}
  \norm{(-\Delta)^{\nu}(E_{0}f(\eps \cdot))}_{H^{\mu}}\lesssim \eps^{2\nu-1/2}\norm{f}_{H^{\mu+2\nu}}.
 \end{equation*}
\end{lemma}
\begin{proof}
The first statement is an immediate consequence of \eqref{eq:frac:Lap:scaling}. Using \cref{Lem:E0,Lem:scaled:Sob:norm} the second claim then follows noting
 \begin{equation*}
  \norm{(-\Delta)^{\nu}(E_{0}f(\eps \cdot))}_{H^{\mu}}\leq \eps^{2\nu-1/2}\norm{(-\Delta)^\nu f}_{H^{\mu+2\nu}}\leq \eps^{2\nu-1/2}\norm{f}_{H^{\mu+2\nu}}.
 \end{equation*}
\end{proof}

The following lemma is a classical result which relates products with Fourier variables with derivatives in spatial variables.

\begin{lemma}\label{Lem:critical:symbol}
For sufficiently regular $f$ we have the following relations:
 \begin{equation*}
  \F^{-1}\Bigl((\xi\mp 1)^2\widehat{f(\eps\cdot)\ee^{\pm\im \cdot}}\Bigr)=-\eps^{2}f''(\eps x)\ee^{\pm\im x}
 \end{equation*}
 as well as
  \begin{equation*}
  \F^{-1}\Bigl((\xi\mp 1)^2\widehat{E_{0}f(\eps\cdot)\ee^{\pm\im \cdot}}\Bigr)=-\eps^{2}E_{0}(f''(\eps \cdot)\ee^{\pm\im \cdot}).
 \end{equation*}
\end{lemma}

Although the Fourier multipliers in \cref{eq:mult:r:pm,eq:mult:m:pm} have no nice representation in terms of spatial derivatives analogously to the previous statement, we still can obtain the following estimates.

\begin{lemma}\label{Lem:est:mult:norm}
 For $f\in H^{\theta+3}$ we have
 \begin{equation*}
  \begin{aligned}
   \norm{\F^{-1}(\mathfrak{r}^{\pm}(\widehat{(E_{0}f(\eps \cdot)\ee^{\pm\im \cdot})})}_{H^{\theta}}&\lesssim \eps^{3-1/2}\norm{f'''}_{H^{\theta}}\\
   \norm{\F^{-1}(m^{1,\pm}(\widehat{(E_{0}f(\eps \cdot)\ee^{\pm 2\im \cdot})})}_{H^{\theta}}&\lesssim \eps^{1/2}\norm{f'}_{H^{\theta}}\\
   \norm{\F^{-1}(m^{2,\pm}(\widehat{(E_{0}f(\eps \cdot)\ee^{\pm 2\im \cdot})})}_{H^{\theta}}&\lesssim \eps^{3-1/2}\norm{f'''}_{H^{\theta}}.
  \end{aligned}
 \end{equation*}
\end{lemma}

\begin{proof}
By means of Lemma~\ref{Lem:est:symbols} we have
 \begin{multline*}
  \norm{\F^{-1}(\mathfrak{r}^{\pm}(\widehat{(E_{0}f(\eps \cdot)\ee^{\pm\im \cdot})})}_{H^{\theta}}=\norm{\mathfrak{r}^{\pm}(\widehat{(E_{0}f(\eps \cdot)\ee^{\pm\im \cdot})}}_{L_{\theta}^{2}}\\*
  =\biggl(\int_{\R}\abs[\Big]{\mathfrak{r}^{\pm}(\xi)m_0(\xi\mp 1)\frac{1}{\eps}\hat{f}\Bigl(\frac{\xi\mp 1}{\eps}\Bigr)}^{2}(1+\xi^{2})^{\theta}\dd{\xi}\biggr)^{1/2}\\*
  \lesssim \biggl(\int_{\R}\abs[\Big]{\frac{1}{\eps}m_0(\xi\mp 1)(\xi\mp 1)^{3}\hat{f}\Bigl(\frac{\xi\mp 1}{\eps}\Bigr)}^{2}(1+\xi^{2})^{\theta}\dd{\xi}\biggr)^{1/2}\\*
  =\eps^{3}\biggl(\int_{\R}\abs[\Big]{m_0(\xi\mp 1)\frac{1}{\eps}\widehat{f'''}\Bigl(\frac{\xi\mp 1}{\eps}\Bigr)}^{2}(1+\xi^{2})^{\theta}\dd{\xi}\biggr)^{1/2}\lesssim \eps^{3}\norm{f'''(\eps \cdot)\ee^{\pm i\cdot}}_{H^{\theta}}\leq \eps^{3-1/2}\norm{f'''}_{H^{\theta}}.
 \end{multline*}
The last step follows from Lemma~\ref{Lem:scaled:Sob:norm}. The other two estimates follow in the same way.
\end{proof}

\section{The residuum}\label{Sec:residuum}

\subsection{Computing the residuum}

Based on the representation for the Fourier symbol of $-(1-(-\Delta)^{\power/2})^2$ in Section~\ref{Sec:approximation:SH-operator} we will compute the residuum (or more precisely the terms up to order $\eps^4$). For this, we will consider the different terms in \eqref{eq:SH:qc} separately and collect them later according to their mode:

First, we get for the time derivative that
\begin{multline*}
 -\del_{t}(\eps\Psi)=-\eps^{3}\bigl(\del_{T}(E_{0}A(\eps x, \eps^{2}t))\ee^{\im x}+\del_{T}(E_{0}\bar{A}(\eps x,\eps^{2} t))\ee^{-\im x}\bigr)\\*
 -\eps^{4}\bigl(\del_{T}(E_{0}A_{2}(\eps x,\eps^{2} t))\ee^{2\im x}+\del_{T}(E_{0}\bar{A}_{2}(\eps x,\eps^{2}t))\ee^{-2\im x}+\del_{T}(E_{0}A_{0}(\eps x,\eps^{2}t))\bigr).
\end{multline*}
Moreover
\begin{multline*}
 \eps^{2}(\eps\Psi(x,t))=\eps^{3}\Bigl((E_{0}A(\eps x,\eps^{2}t))\ee^{\im x}+(E_{0}\bar{A}(\eps x,\eps^{2}t))\ee^{-\im x}\Bigr)\\*
 +\eps^{4}\Bigl((E_{0}A_2(\eps x,\eps^{2}t))\ee^{2\im x}+(E_{0}\bar{A}_2(\eps x,\eps^{2}t))\ee^{-2\im x}+(E_{0}A_0(\eps x,\eps^{2}t))\Bigr).
\end{multline*}
According to \cref{Lem:Taylor:symbol,Lem:critical:symbol,Lem:remainder:at:two} we have
\begin{multline*}
 -(1-(-\Delta)^{\power/2})^{2}(\eps \Psi)=\power^{2}\eps^{3}\bigl((E_{0}A''(\eps x))\ee^{\im x}+(E_{0}\bar{A}''(\eps x))\ee^{-\im x}\bigr)\\*
 -\eps\F^{-1}\bigl(R^{+}(\xi)\widehat{(E_{0}A(\eps \cdot))\ee^{\im \cdot}}+R^{-}(\xi)\widehat{(E_{0}\bar{A}(\eps \cdot))\ee^{-\im \cdot}}\bigr)\\*
-\power^{2}\eps^{2}(-\del_{x}^{2}\pm 2\im \del_{x}+1)\bigl((E_{0}A_{2}(\eps x))\ee^{2\im x} +(E_{0}\bar{A}_{2}(\eps x))\ee^{-2\im x}\bigr)-(1-(-\Delta)^{\power/2})^{2}(\eps^{2}(E_{0}A_{0})(\eps x))\\*
-\eps^{2}\Bigl(\mathfrak{c}^{+}(E_{0}A_{2}(\eps x))\ee^{2\im x}+\mathfrak{c}^{-}(E_{0}\bar{A}_{2}(\eps x))\ee^{-2\im x}+M^{1,+}\bigl((E_{0}A_{2}(\eps x))\ee^{2\im x}\bigr)+M^{1,-}\bigl(E_{0}\bar{A}_{2}(\eps x)\bigr)\ee^{-2\im x}\bigr)\\*
+M^{2,+}\bigl((E_{0}A_{2}(\eps x))\ee^{2\im x}\bigr)+M^{2,-}\bigl((E_{0}\bar{A}_{2}(\eps x))\ee^{-2\im x}\bigr)\Bigr).
\end{multline*}
Expanding further we get
\begin{multline*}
-(1-(-\Delta)^{\power/2})^{2}(\eps \Psi)=\power^{2}\eps^{3}\bigl((E_{0}A''(\eps x))\ee^{\im x}+(E_{0}\bar{A}''(\eps x))\ee^{-\im x}\bigr)\\*
 -\eps\F^{-1}\bigl(R^{+}(\xi)\widehat{(E_{0}A(\eps \cdot))\ee^{\im \cdot}}+R^{-}(\xi)\widehat{(E_{0}\bar{A}(\eps \cdot))\ee^{-\im \cdot}}\bigr)\\*
+\Bigl(\power^{2}\eps^{4}(E_{0}A_{2}''(\eps x))\ee^{2\im x}+4s^{2}\im\eps^{3}(E_{0}A_{2}'(\eps x))\ee^{2\im x}-4s^2 \eps^{2}(E_{0}A_{2}(\eps x))\ee^{2\im x}\\*
-2\power^2 \im \eps^{3}(E_{0}A_{2}'(\eps x))\ee^{2\im x}+4\power^{2}\eps^{2}(E_{0}A_{2}(\eps x))\ee^{2\im x}-\power^{2}\eps^{2}(E_{0}A_{2}(\eps x))\ee^{2\im x}\\*
+\power^{2}\eps^{4}(E_{0}\bar{A}_{2}''(\eps x))\ee^{-2\im x}-4\power^{2}\im \eps^{3}(E_{0}\bar{A}_{2}'(\eps x))\ee^{-2\im x}-4\power^{2}\eps^{2}(E_{0}\bar{A}_{2}(\eps x))\ee^{-2\im x}\\*
+2\power^{2}\im \eps^{3}(E_{0}\bar{A}_{2}'(\eps x))\ee^{-2\im x}+4\power^{2}\eps^{2}(E_{0}\bar{A}_{2}(\eps x))\ee^{-2\im x}-\power^{2}\eps^{2}(E_{0}\bar{A}_{2}(\eps x))\ee^{-2\im x}\Bigr)\\*
-\Bigl(\eps^{2}(E_{0}A_{0}(\eps x))-2\eps^{2+\power}(E_{0}((-\Delta)^{\power/2}A_{0})(\eps x))+\eps^{2+2\power}(E_{0}((-\Delta)^{\power}A_{0})(\eps x))\Bigr)\\*
-\eps^{2}\Bigl(\mathfrak{c}^{+}(E_{0}A_{2}(\eps x))\ee^{2\im x}+\mathfrak{c}^{-}(E_{0}\bar{A}_{2}(\eps x))\ee^{-2\im x}+M^{1,+}\bigl((E_{0}A_{2}(\eps x))\ee^{2\im x}\bigr)+M^{1,-}\bigl((E_{0}\bar{A}_{2}(\eps x))\ee^{-2\im x}\bigr)\\*
+M^{2,+}\bigl((E_{0}A_{2}(\eps x))\ee^{2\im x}\bigr)+M^{2,-}\bigl((E_{0}\bar{A}_{2}(\eps x))\ee^{-2\im x}\bigr)\Bigr).
\end{multline*}
Combining, we find
\begin{multline*}
-(1-(-\Delta)^{\power/2})^{2}(\eps \Psi)\\*
 \shoveleft{=\power^{2}\eps^{3}\bigl((E_{0}A''(\eps x))\ee^{\im x}+(E_{0}\bar{A}''(\eps x))\ee^{-\im x}\bigr)-\eps\F^{-1}\bigl(R^{+}(\xi)\widehat{(E_{0}A(\eps \cdot))\ee^{\im \cdot}}+R^{-}(\xi)\widehat{(E_{0}\bar{A}(\eps \cdot))\ee^{-\im \cdot}}\bigr)}\\*
+\Bigl(\power^{2}\eps^{4}(E_{0}A_{2}''(\eps x))\ee^{2\im x}+2\power^{2}\im\eps^{3}(E_{0}A_{2}'(\eps x))\ee^{2\im x}-\power^2 \eps^{2}(E_{0}A_{2}(\eps x))\ee^{2\im x}\\*
+\power^{2}\eps^{4}(E_{0}\bar{A}_{2}''(\eps x))\ee^{-2\im x}-2\power^{2}\im \eps^{3}(E_{0}\bar{A}_{2}(\eps x))\ee^{-2\im x}-\power^{2}\eps^{2}(E_{0}\bar{A}_{2}(\eps x))\ee^{-2\im x}\Bigr)\\*
-\Bigl(\eps^{2}(E_{0}A_{0}(\eps x))-2\eps^{2+\power}(E_{0}((-\Delta)^{\power/2}A_{0})(\eps x))+\eps^{2+2\power}(E_{0}((-\Delta)^{\power}A_{0})(\eps x))\Bigr)\\*
-\eps^{2}\Bigl(\mathfrak{c}^{+}(E_{0}A_{2}(\eps x))\ee^{2\im x}+\mathfrak{c}^{-}(E_{0}\bar{A}_{2}(\eps x))\ee^{-2\im x}+M^{1,+}\bigl((E_{0}A_{2}(\eps x))\ee^{2\im x}\bigr)+M^{1,-}\bigl((E_{0}\bar{A}_{2}(\eps x))\ee^{-2\im x}\bigr)\\*
+M^{2,+}\bigl((E_{0}A_{2}(\eps x))\ee^{2\im x}\bigr)+M^{2,-}\bigl((E_{0}\bar{A}_{2}(\eps x))\ee^{-2\im x}\bigr)\Bigr).
\end{multline*}
It remains to compute the non-linear terms:
\begin{multline*}
 a_{1}(\eps \Psi)^{2}\\*
 \shoveleft{=a_{1}\eps^{2}\Bigl((E_{0}A)^{2}\ee^{2\im x}+2(E_{0}A)(E_{0}\bar{A})+(E_{0}\bar{A})^{2}\ee^{-2\im x}\Bigr)}\\*
 +2a_{1}\eps^{3}\Bigl((E_{0}A)(E_{0}A_{2})\ee^{3\im x}+(E_{0}A)(E_{0}\bar{A}_{2})\ee^{-\im x}+(E_{0}A)(E_{0}A_{0})\ee^{\im x}+(E_{0}\bar{A})(E_{0}A_{2})\ee^{\im x}\\*
 +(E_{0}\bar{A})(E_{0}\bar{A}_{2})\ee^{-3\im x}+(E_{0}\bar{A})(E_{0}A_{0})\ee^{-\im x}\Bigr)\\*
 +a_{1}\eps^{4}\Bigl((E_{0}A_{2})\ee^{2\im x}+(E_{0}\bar{A}_{2})\ee^{-2\im x}+(E_{0}A_{0})\Bigr)^{2}.
\end{multline*}
Similarly
\begin{multline*}
 a_{2}(\eps \Psi)^{3}\\*
 \shoveleft{=a_{2}\eps^{3}\Bigl((E_{0}A)^{3}\ee^{3\im x}+3(E_{0}A)^{2}(E_{0}\bar{A})\ee^{\im x}+3(E_{0}A)(E_{0}\bar{A})^{2}\ee^{-\im x}+(E_{0}\bar{A})^{3}\ee^{-3\im x}\Bigr)}\\*
 +3a_{2}\eps^{4}\Bigl((E_{0}A)^{2}(E_{0}A_{2})\ee^{4\im x}+2(E_{0}A)(E_{0}\bar{A})(E_{0}A_{2})\ee^{2\im x}+(E_{0}\bar{A})^{2}(E_{0}A_{2})+(E_{0}A)^{2}(E_{0}\bar{A}_{2})\\*
 +2(E_{0}A)(E_{0}\bar{A})(E_{0}\bar{A}_{2})\ee^{-2\im x}+(E_{0}\bar{A})^{2}(E_{0}\bar{A}_{2})\ee^{-4\im x}+(E_{0}A)^{2}(E_{0}A_{0})\ee^{2\im x}\\*
 +2(E_{0}A)(E_{0}\bar{A})(E_{0}A_{0})+(E_{0}\bar{A})^{2}(E_{0}A_{0})\ee^{-2\im x}\Bigr)+\Ord(\eps^{5}).
\end{multline*}
In summary, we can write the residuum in the following form
\begin{equation*}
 \Res(\eps \Psi)=\sum_{k=-4}^{4}z_{k}\ee^{-k\im x}+\Ord(\eps^{5})
\end{equation*}
with
\begin{equation*}
 \begin{aligned}
  z_{0}&=-\eps^{4}\del_{T}(E_{0}A_{0}(\eps x))+\eps^{4}(E_{0}A_{0}(\eps x))-\eps^{2}(E_{0}A_{0}(\eps x))+2\eps^{2+\power}(E_{0}((-\Delta)^{\power/2}A_{0})(\eps x))\\
  &\quad-\eps^{2+2\power}(E_{0}((-\Delta)^{\power}A_{0})(\eps x))-2a_{1}\eps^{2}(E_{0}A(\eps x))(E_{0}\bar{A}(\eps x))-2a_{1}\eps^{4}(E_{0}A_{2}(\eps x))(E_{0}\bar{A}_{2}(\eps x))\\
  &\quad-a_{1}\eps^{4}(E_{0}A_{0}(\eps x))^{2}(\eps x)-3a_{2}\eps^{4}(E_{0}\bar{A}(\eps x))^{2}(E_{0}A_{2}(\eps x))-3a_{2}\eps^{4}(E_{0}A(\eps x))^{2}(E_{0}\bar{A}_{2}(\eps x))\\
  &\quad-6a_{2}\eps^{4}(E_{0}A(\eps x))(E_{0}\bar{A}(\eps x))(E_{0}A_{0}(\eps x))\\
  z_{1}&=-\eps^{3}\del_{T}(E_{0}A(\eps x))+\eps^{3}(E_{0}A(\eps x))+\power^{2}\eps^{3}(E_{0}A''(\eps x))-\eps\F^{-1}(R^{+}(\xi)\widehat{(E_{0}A(\eps\cdot))\ee^{\im \cdot}})\ee^{-\im x}\\
  &\quad-2a_{1}\eps^{3}(E_{0}A(\eps x))(E_{0}A_{0}(\eps x))-2a_1\eps^3(E_{0}\bar{A}(\eps x))(E_{0}A_{2}(\eps x))-3a_{2}\eps^{3}(E_{0}A(\eps x))^{2}(E_{0}\bar{A}(\eps x))\\
  z_{2}&=-\eps^{4}\del_{T}(E_{0}A_{2}(\eps x))+\eps^{4}(E_{0}A_{2}(\eps x))+\power^{2}\eps^{4}(E_{0}A_{2}''(\eps x))+2\im \power^{2}\eps^{3}(E_{0}A_{2}'(\eps x))\\
  &\quad-\power^{2}\eps^{2}(E_{0}A_{2}(\eps x))-\mathfrak{c}^{+}\eps^{2}(E_{0}A_{2}(\eps x))-\eps^{2}M^{1,+}(E_{0}A_{2}(\eps\cdot)\ee^{\im \cdot})\ee^{-2\im x}-\eps^{2}M^{2,+}(E_{0}A_{2}(\eps\cdot)\ee^{2\im \cdot})\ee^{-2\im x}\\
  &\quad -a_{1}\eps^{2}(E_{0}A(\eps x))^{2}-2a_{1}\eps^{4}(E_{0}A_{2}(\eps x))(E_{0}A_{0}(\eps x))\\
  &\quad-6a_{2}\eps^{4}(E_{0}A(\eps x))(E_{0}\bar{A}(\eps x))(E_{0}A_{2}(\eps x))-3a_{2}\eps^{4}(E_{0}A(\eps x))^{2}(E_{0}A_{0}(\eps x))\\
  z_{3}&=-2a_{1}\eps^{3}(E_{0}A(\eps x))(E_{0}A_{2}(\eps x))-a_{2}\eps^{3}(E_{0}A(\eps x))^{3}\\
  z_{4}&=-a_{1}\eps^{4}(E_{0}A_{2}(\eps x))^{2}-3a_{2}(E_{0}A(\eps x))^{2}(E_{0}A_{2}(\eps x))\\
  z_{-k}&=\bar{z}_{k} \qquad \text{for }k=1\ldots 4.
 \end{aligned}
\end{equation*}

\subsection{Estimating the residuum}

We will now verify the estimates \ref{It:ass:5} on the residuum which will conclude the proof of our main result Theorem~\ref{Thm:main}. For this, we note that we can neglect the exponential factors since they only result in a shift on the Fourier side, which does not affect the norm. It thus suffices to consider the coefficients $z_{k}$. We also note that from the choice of $m_c$ and $m_0$, the terms $z_{\pm1}$ correspond to the critical modes while the remaining ones are stable, i.e.\@ $E_c \Res(\eps \Psi)=z_1\ee^{\im x}+z_{-1}\ee^{-\im x}$ and $E_s(\eps\Psi)=\Res(\eps\Psi)-(z_1\ee^{\im x}+z_{-1}\ee^{-\im x})$.

\subsubsection{The critical modes -- estimating $z_{1}$}

We first rewrite
\begin{multline*}
 z_{1}=\eps^{3}\Bigl[\del_{T}(E_{0}^{c}A(\eps x))-\del_{T}A(\eps x,\eps^{2} t)-(E_{0}^{c}A(\eps x))+A(\eps x)-\power^{2}(E_{0}^{c}A''(\eps x))+\power^{2}A''(\eps x)\\*
  +2a_{1}(E_{0}^{c}A(\eps x))(E_{0}A_{0}(\eps x))+2a_{1}A(E_{0}^{c}A_{0}(\eps x))+4a_{1}^2\abs{A}^{2}A+2a_1(E_0^c\bar{A}(\eps x))(E_0 A_2(\eps x))\\*
 +2a_1\bar{A}(E_0^cA_2(\eps x))+\frac{2a_1^2}{\power^2+\mathfrak{c}^{+}}\abs{A}^2A+3a_{2}(E_{0}^{c}A(\eps x))(E_{0}A(\eps x))(E_{0}\bar{A}(\eps x))\\*
  +3a_{2}A(E_{0}^{c}A)(E_{0}\bar{A})+3a_{2}A^{2}(E_{0}^{c}\bar{A})-3a_{2}\abs{A}^{2}A\Bigr]\\*
 -\eps\F^{-1}(R^{+}(\xi)\widehat{(E_{0}A(\eps\cdot))\ee^{\im \cdot}})\ee^{-\im x}.
 \end{multline*}
Thus, due to \eqref{eq:GL}, it remains
\begin{multline*}
 z_{1}=\eps^{3}\Bigl[\del_{T}(E_{0}^{c}A(\eps x))-(E_{0}^{c}A(\eps x))+A(\eps x)-\power^{2}(E_{0}^{c}A''(\eps x))+2a_{1}(E_{0}^{c}A(\eps x))(E_{0}A_{0}(\eps x))\\*
 +2a_{1}A(E_{0}^{c}A_{0})+2a_1(E_0^c\bar{A}(\eps x))(E_0 A_2(\eps x))+2a_1\bar{A}(E_0^cA_2(\eps x))\\* 
 +3a_{2}(E_{0}^{c}A(\eps x))(E_{0}A(\eps x))(E_{0}\bar{A}(\eps x))+3a_{2}A(E_{0}^{c}A)(E_{0}\bar{A})+3a_{2}A^{2}(E_{0}^{c}\bar{A})\Bigr]\\*
 -\eps\F^{-1}(R^{+}(\xi)\widehat{(E_{0}A(\eps\cdot))\ee^{\im \cdot}})\ee^{-\im x}.
 \end{multline*}
To estimate $\del_{T}(E_{0}^{c}A)$ we rely again on \eqref{eq:GL}, i.e.\@ we can replace $\del_{T}A$ by the right-hand side of \eqref{eq:GL}. Together with \cref{Lem:E0c,Lem:Sob:norm:prod,Lem:L1:Htheta} we thus get
\begin{multline*}
 \sup_{t\in[0,T_{*}/\eps^{2}]}\norm{\del_{T}E_{0}^{c}(A(\eps x,\eps^{2}t))}_{H^{\theta}}\lesssim \eps^{\theta-1/2}\Bigl[\norm{A}_{H^{\theta}}+\norm{A''}_{H^{\theta}}+\norm{\abs{A}^{2}A}_{H^{\theta}}\Bigr]\\*
 \leq \eps^{\theta-1/2}\Bigl[\norm{A}_{H^{\theta+2}}+\norm{A}_{H^{\theta}}^{3}\Bigr].
\end{multline*}
Recalling also from \eqref{eq:choice:of:parameters} that $A_{0}=-2a_1\abs{A}^{2}=-2a_1A\bar{A}$, $A_2=-\frac{a_1}{\power^2+\mathfrak{c}^{+}}A^2$ and using additionally Lemma~\ref{Lem:E0c} we deduce similarly
\begin{equation*}
 \begin{aligned}
  \sup_{t\in[0,T_{*}/\eps^{2}]}\norm{E_{0}^{c}(A(\eps x,\eps^{2}t))}_{H^{\theta}}&\lesssim \eps^{\theta-1/2}\norm{A}_{H^{\theta}}\\
  \sup_{t\in[0,T_{*}/\eps^{2}]}\norm{(E_{0}^{c}A''(\eps x,\eps^{2}t))}_{H^{\theta}}&\lesssim \eps^{\theta-1/2}\norm{A}_{H^{\theta+2}}\\
  \sup_{t\in[0,T_{*}/\eps^{2}]}\norm{(E_{0}^{c}A(\eps x))(E_{0}A_{0}(\eps x))+A(E_{0}^{c}A_{0})}_{H^{\theta}}&\lesssim \eps^{\theta-1/2}\norm{A}_{H^{\theta+1}}^{3}\\
  \sup_{t\in[0,T_{*}/\eps^{2}]}\norm{(E_{0}^{c}\bar{A}(\eps x))(E_{0}A_{2}(\eps x))+\bar{A}(E_{0}^{c}A_{2})}_{H^{\theta}}&\lesssim \eps^{\theta-1/2}\norm{A}_{H^{\theta+1}}^{3}\\
  \sup_{t\in[0,T_{*}/\eps^{2}]}\norm{(E_{0}^{c}A(\eps x))(E_{0}A(\eps x))(E_{0}\bar{A}(\eps x))+A(E_{0}^{c}A)(E_{0}\bar{A})+A^{2}(E_{0}^{c}\bar{A})}_{H^{\theta}}&\lesssim \eps^{\theta-1/2}\norm{A}_{H^{\theta+1}}^{3}.
   \end{aligned}
\end{equation*}
Finally, from Lemma~\ref{Lem:est:mult:norm} we have
\begin{equation*}
 \sup_{t\in[0,T_{*}/\eps^{2}]}\norm{\F^{-1}(R^{+}(\xi)\widehat{(E_{0}A(\eps\cdot))\ee^{\im \cdot}})\ee^{-\im x}}_{H^{\theta}}\lesssim \eps^{3-1/2}\norm{A'''}_{H^{\theta}}.
\end{equation*}
In summary this yields
\begin{equation*}
 \sup_{t\in[0,T_{*}/\eps^{2}]}\norm{z_{1}}_{H^{\theta}}\lesssim (\eps^{3+\theta-1/2}+\eps^{4-1/2})(\norm{A}_{H^{\theta+3}}+\norm{A}_{H^{\theta+1}}^{3})\lesssim \eps^{\beta+2}(\norm{A}_{H^{\theta+3}}+\norm{A}_{H^{\theta+1}}^{3})
\end{equation*}
if $\theta\geq 1$ while we recall that $\beta=3/2$.

\subsubsection{The stable modes}

\paragraph*{Estimating $z_{0}$:}

Taking into account \cref{Lem:E0,eq:GL,eq:choice:of:parameters} we obtain similarly as for $z_{1}$ that 
\begin{equation*}
 \norm{\del_{T}(E_{0}A_{0}(\eps \cdot))}_{H^{\theta}}\lesssim \eps^{-1/2}(\norm{A}_{H^{\theta+2}}+\norm{A}_{H^{\theta}}^{3})\norm{A}_{H^{\theta}}
\end{equation*}
Moreover
\begin{equation*}
 \begin{aligned}
  \norm{E_{0}A_{0}(\eps \cdot)}_{H^{\theta}}&\lesssim \eps^{-1/2}\norm{A}_{H^{\theta}}^{2}\\
  \norm{E_{0}((-\Delta)^{\power/2}A_{0})(\eps \cdot)}_{H^{\theta}}&\lesssim\eps^{-1/2}\norm{A}_{H^{\theta+s}}^{2}\\
  \norm{E_{0}((-\Delta)^{\power}A_{0})(\eps \cdot)}_{H^{\theta}}&\lesssim\eps^{-1/2}\norm{A}_{H^{\theta+2s}}^{2}\\
  \norm{(E_{0}A_{2}(\eps x))(E_{0}\bar{A}_{2}(\eps x))}_{H^{\theta}}&\lesssim \eps^{-1/2}\norm{A}_{H^{\theta}}^{4}\\
  \norm{(E_{0}A_{0}(\eps x))^{2}(\eps x)}_{H^{\theta}}&\lesssim \eps^{-1/2}\norm{A}_{H^{\theta}}^{2}\\
  \norm{(E_{0}\bar{A}(\eps x))^{2}(E_{0}A_{2}(\eps x))}_{H^{\theta}}&\lesssim \eps^{-1/2}\norm{A}_{H^{\theta}}^{3}\\
  \norm{(E_{0}\bar{A}(\eps x))^{2}(E_{0}A_{2}(\eps x))}_{H^{\theta}}&\lesssim \eps^{-1/2}\norm{A}_{H^{\theta}}^{3}\\
  \norm{(E_{0}A(\eps x))^{2}(E_{0}\bar{A}_{2}(\eps x))}_{H^{\theta}}&\lesssim \eps^{-1/2}\norm{A}_{H^{\theta}}^{3}
 \end{aligned}
\end{equation*}
Finally, recalling also \eqref{eq:choice:of:parameters} we rewrite
\begin{multline*}
 -(E_{0}A_{0}(\eps x)-2a_{1}(E_{0}A(\eps x))(E_{0}\bar{A}(\eps x))=2a_1(E_{0}(A\bar{A}(\eps x)))-2a_{1}(E_{0}A(\eps x))(E_{0}\bar{A}(\eps x))\\*
 =-2a_{1}(E_{0}^{c}\abs{A}^{2}(\eps x))+2a_{1}(E_{0}^{c}A(\eps x))\bar{A}(\eps x)+2a_{1}(E_{0}A(\eps x))(E_{0}^{c}\bar{A}(\eps x)).
\end{multline*}
Together with \cref{Lem:E0c,Lem:Sob:norm:prod} we thus conclude
\begin{equation*}
 \norm{-(E_{0}A_{0}(\eps x)-2a_{1}(E_{0}A(\eps x))(E_{0}\bar{A}(\eps x))}_{H^{\theta}}\lesssim \eps^{\theta-1/2}\norm{A}_{H^{\theta+1}}^{2}.
\end{equation*}
In summary
\begin{equation*}
 \begin{split}
   \norm{z_{0}}_{H^{\theta}}&\lesssim \bigl(\eps^{4+\theta-1/2}+\eps^{4-1/2}+\eps^{2+\power-1/2}+\eps^{2+\theta-1/2}\bigr)\bigl[\norm{A}_{H^{\theta+2}}^{2}+\norm{A}_{H^{\theta+2\power}}^{2}\bigr](1+\norm{A}_{H^{\theta}}^{2})\\
   &\lesssim \eps^{\beta+1}\bigl[\norm{A}_{H^{\theta+2}}^{2}+\norm{A}_{H^{\theta+2\power}}^{2}\bigr](1+\norm{A}_{H^{\theta}}^{2})
 \end{split}
\end{equation*}
if $\power,\theta\geq 1$ since $\beta=3/2$.

\paragraph*{Estimating $z_{2}$:}

We proceed as before, i.e.\@
\begin{equation*}
 \norm{\del_{T}(E_{0}A_{2}(\eps \cdot))}_{H^{\theta}}\lesssim \eps^{-1/2}(\norm{A}_{H^{\theta+2}}+\norm{A}_{H^{\theta}}^{3})\norm{A}_{H^{\theta}}
\end{equation*}
Moreover
\begin{equation*}
 \begin{aligned}
  \norm{E_{0}A_{2}(\eps \cdot)}_{H^{\theta}}&\lesssim \eps^{-1/2}\norm{A}_{H^{\theta}}^{2}\\
  \norm{E_{0}A_2''(\eps \cdot)}_{H^{\theta}}&\lesssim\eps^{-1/2}\norm{A}_{H^{\theta+2}}^{2}\\
  \norm{E_{0}A_2'(\eps \cdot)}&\lesssim\eps^{-1/2}\norm{A}_{H^{\theta+1}}^{2}\\
  \norm{(E_{0}A_{2}(\eps x))(E_{0}A_{0}(\eps x))}_{H^{\theta}}&\lesssim \eps^{-1/2}\norm{A}_{H^{\theta}}^{3}\\
  \norm{(E_{0}A(\eps x))(E_{0}\bar{A}(\eps x))(E_{0}A_{2}(\eps x))}_{H^{\theta}}&\lesssim \eps^{-1/2}\norm{A}_{H^{\theta}}^{4}\\
  \norm{(E_{0}A(\eps x))^{2}(E_{0}A_{0}(\eps x))}_{H^{\theta}}&\lesssim \eps^{-1/2}\norm{A}_{H^{\theta}}^{3}\\
 \end{aligned}
\end{equation*}
 Furthermore, from Lemma~\ref{Lem:est:mult:norm}
 \begin{equation*}
  \begin{aligned}
  \norm{M^{1,+}(E_{0}A_{2}(\eps\cdot))\ee^{-2\im x}}_{H^{\theta}}&\lesssim \eps^{1-1/2}\norm{A}_{H^{\theta+1}}^{2}\\
  \norm{M^{2,+}(E_{0}A_{2}(\eps\cdot))\ee^{-2\im x}}_{H^{\theta}}&\lesssim \eps^{3-1/2}\norm{A}_{H^{\theta+3}}^{2}
  \end{aligned}
 \end{equation*}
 Finally, for the terms of $\Ord(\eps^{2})$ we again rewrite taking also \eqref{eq:choice:of:parameters} into account which yields
 \begin{multline*}
  -\power^{2}(E_{0}A_{2}(\eps x))-\mathfrak{c}^{+}(E_{0}A_{2}(\eps x))-a_{1}(E_{0}A(\eps X))^{2}\\*
  =-a_1(E_{0}^{c}A^{2}(\eps x))+a_1\Bigl(A-E_{0}A\Bigr)A+a_{1}(E_{0}A)A-a_{1}(E_{0}A)^{2}\\*
  =-a_{1}(E_{0}^{c}A^{2}(\eps x))+a_{1}(E_{0}^{c}A(\eps x))A(\eps x)+a_{1}(E_{0}A(\eps x))(E_{0}^{c}A(\eps x)).
 \end{multline*}
Thus,
\begin{multline*}
 \norm{-\power^{2}(E_{0}A_{2}(\eps x))-\mathfrak{c}^{+}(E_{0}A_{2}(\eps x))-a_{1}(E_{0}A(\eps X))^{2}}_{H^{\theta}}\\*
 \leq \norm{-a_{1}(E_{0}^{c}A^{2}(\eps x))+a_{1}(E_{0}^{c}A(\eps x))A(\eps x)+a_{1}(E_{0}A(\eps x))(E_{0}^{c}A(\eps x))}_{H^{\theta}}\lesssim \eps^{\theta-1/2}\norm{A}_{H^{\theta+1}}^{2}.
\end{multline*}
In summary
\begin{equation*}
 \norm{z_{2}}_{H^{\theta}}\lesssim \bigl(\eps^{4-1/2}+\eps^{3-1/2}+\eps^{2+\theta-1/2}\bigr)\bigl[\norm{A}_{H^{\theta+3}}^{2}+\norm{A}_{H^{\theta}}^{4}\bigr]\lesssim \eps^{\beta+1}\bigl[\norm{A}_{H^{\theta+3}}^{2}+\norm{A}_{H^{\theta}}^{4}\bigr]
\end{equation*}
if $\theta\geq 1$ since $\beta=3/2$.

 \paragraph*{Estimating $z_3$ and $z_4$:}
 
 Proceeding as above, we get
 \begin{equation*}
  \begin{aligned}
   \norm{z_{3}}_{H^{\theta}}&\lesssim \eps^{3-1/2}\norm{A}_{H^{\theta}}^{3}\lesssim \eps^{\beta+1}\norm{A}_{H^{\theta}}^{3}\\
   \norm{z_{4}}_{H^{\theta}}&\lesssim \eps^{4-1/2}\norm{A}_{H^{\theta}}^{3}\lesssim \eps^{\beta+1}\norm{A}_{H^{\theta}}^{3}
  \end{aligned}
 \end{equation*}
 since $\beta=3/2$.\medskip
 
 \textbf{Acknowledgments:} CK and ST would like to thank Cinzia Soresina for interesting discussions regarding space-fractional PDEs. CK and ST would like to thank the VolkswagenStiftung for support via a Lichtenberg Professorship awarded to CK.


\begin{thebibliography}{10}

\bibitem{Achleitneretal}
F.~Achleitner, G.~Akagi, C.~Kuehn, J.M. Melenk, J.D.M. Rademacher, C.~Soresina,
  and J.~Yang.
\newblock Fractional dissipative pdes.
\newblock {\em arXiv}, pages 1--63, 2023.

\bibitem{AchleitnerKuehn1}
F.~Achleitner and C.~Kuehn.
\newblock Traveling waves for a bistable equation with nonlocal-diffusion.
\newblock {\em Adv. Differential Equat.}, 20(9):887--936, 2015.

\bibitem{AKMR21}
Franz Achleitner, Christian Kuehn, Jens~M. Melenk, and Alexander Rieder.
\newblock Metastable speeds in the fractional {A}llen-{C}ahn equation.
\newblock {\em Appl. Math. Comput.}, 408:Paper No. 126329, 18, 2021.

\bibitem{BeH21}
A.~Behzadan and M.~Holst.
\newblock Multiplication in {S}obolev spaces, revisited.
\newblock {\em Ark. Mat.}, 59(2):275--306, 2021.

\bibitem{bucur2016nonlocal}
Claudia Bucur, Enrico Valdinoci, et~al.
\newblock {\em Nonlocal diffusion and applications}, volume~20.
\newblock Springer, 2016.

\bibitem{BurrageHaleKay}
K.~Burrage, N.~Hale, and D.~Kay.
\newblock An efficient implicit {FEM} scheme for fractional-in-space
  reaction-diffusion equations.
\newblock {\em SIAM J. Sci. Comput.}, 34(4):A2145--A2172, 2012.

\bibitem{CabreRoquejoffre}
X.~Cabr{\'e} and J.M. Roquejoffre.
\newblock The influence of fractional diffusion in {Fisher-KPP} equations.
\newblock {\em Commun. Math. Phys.}, 320(3):679--722, 2013.

\bibitem{CaffarelliSilvestre}
L.~Caffarelli and L.~Silvestre.
\newblock An extension problem related to the fractional {Laplacian}.
\newblock {\em Commun. Partial Differential Equat.}, 32(8):1245--1260, 2007.

\bibitem{ChenKimSong}
Z.-Q. Chen, P.~Kim, and R.~Song.
\newblock Heat kernel estimates for the {Dirichlet} fractional {Laplacian}.
\newblock {\em J. Eur. Math. Soc.}, 12:1307--1329, 2010.

\bibitem{ColletEckmann1}
P.~Collet and J.P. Eckmann.
\newblock The time dependent amplitude equation for the {Swift-Hohenberg}
  problem.
\newblock {\em Comm. Math. Phys.}, 132(1):139--153, 1990.

\bibitem{CrossHohenberg}
M.C. Cross and P.C. Hohenberg.
\newblock Pattern formation outside of equilibrium.
\newblock {\em Rev. Mod. Phys.}, 65(3):851--1112, 1993.

\bibitem{del-Castillo-NegreteCarrerasLynch}
D.~del Castillo-Negrete, B.A. Carreras, and V.E. Lynch.
\newblock Front dynamics in reaction-diffusion systems with {Levy} flights: a
  fractional diffusion approach.
\newblock {\em Phys. Rev. Lett.}, 91(1):018302, 2003.

\bibitem{DPV12}
Eleonora Di~Nezza, Giampiero Palatucci, and Enrico Valdinoci.
\newblock Hitchhiker's guide to the fractional {S}obolev spaces.
\newblock {\em Bull. Sci. Math.}, 136(5):521--573, 2012.

\bibitem{EhstandKuehnSoresina}
N.~Ehstand, C.~Kuehn, and C.~Soresina.
\newblock Numerical continuation for fractional {PDEs}: sharp teeth and bloated
  snakes.
\newblock {\em Comm. Nonl. Sci. Numer. Simul.}, page 105762, 2021.

\bibitem{Eva10}
Lawrence~C. Evans.
\newblock {\em Partial differential equations}, volume~19 of {\em Graduate
  Studies in Mathematics}.
\newblock American Mathematical Society, Providence, RI, second edition, 2010.

\bibitem{faustmann-melenk-praetorius21}
Markus Faustmann, Jens~Markus Melenk, and Dirk Praetorius.
\newblock Quasi-optimal convergence rate for an adaptive method for the
  integral fractional {L}aplacian.
\newblock {\em Math. Comp.}, 90(330):1557--1587, 2021.

\bibitem{golovin2008turing}
Alexander~A Golovin, Bernard~J Matkowsky, and Vladimir~A Volpert.
\newblock Turing pattern formation in the {B}russelator model with
  superdiffusion.
\newblock {\em SIAM Journal on Applied Mathematics}, 69(1):251--272, 2008.

\bibitem{Hoyle}
R.~Hoyle.
\newblock {\em Pattern Formation: An Introduction to Methods}.
\newblock Cambridge University Press, 2006.

\bibitem{KSM92}
Pius Kirrmann, Guido Schneider, and Alexander Mielke.
\newblock The validity of modulation equations for extended systems with cubic
  nonlinearities.
\newblock {\em Proc. Roy. Soc. Edinburgh Sect. A}, 122(1-2):85--91, 1992.

\bibitem{KuehnBook1}
C.~Kuehn.
\newblock {\em PDE Dynamics: An Introduction}.
\newblock SIAM, 2019.

\bibitem{KuehnThrom}
C.~Kuehn and S.~Throm.
\newblock Validity of amplitude equations for non-local non-linearities.
\newblock {\em J. Math. Phys.}, 59:071510, 2018.

\bibitem{MainardiLuchkoPagnini}
F.~Mainardi, Y.~Luchko, and G.~Pagnini.
\newblock The fundamental solution of the space-time fractional diffusion
  equation.
\newblock {\em Fract. Calc. Appl. Anal.}, 4(2):153--192, 2001.

\bibitem{MetzlerKlafter}
R.~Metzler and J.~Klafter.
\newblock The random walk's guide to anomalous diffusion: a fractional dynamics
  approach.
\newblock {\em Phys. Rep.}, 339(1):1--77, 2000.

\bibitem{MorganDawes}
D.~Morgan and J.H.P. Dawes.
\newblock The {Swift-Hohenberg} equation with a nonlocal nonlinearity.
\newblock {\em Physica D}, 270:60--80, 2014.

\bibitem{Sch94b}
G.~Schneider.
\newblock A new estimate for the {G}inzburg-{L}andau approximation on the real
  axis.
\newblock {\em J. Nonlinear Sci.}, 4(1):23--34, 1994.

\bibitem{ScU17}
Guido Schneider and Hannes Uecker.
\newblock {\em Nonlinear {PDE}s}, volume 182 of {\em Graduate Studies in
  Mathematics}.
\newblock American Mathematical Society, Providence, RI, 2017.
\newblock A dynamical systems approach.

\bibitem{servadei2014spectrum}
R.~Servadei and E.~Valdinoci.
\newblock On the spectrum of two different fractional operators.
\newblock {\em Proceedings of the Royal Society of Edinburgh Section A:
  Mathematics}, 144(4):831--855, 2014.

\bibitem{vanHarten}
A.~van Harten.
\newblock On the validity of the {Ginzburg-Landau} equation.
\newblock {\em J. Nonlinear Sci.}, 1(4):397--422, 1991.

\end{thebibliography}
\end{document}